\documentclass[11pt,twoside,a4paper]{article}
\usepackage{color,amscd,amsmath,amssymb,amsthm,latexsym,stmaryrd,authblk,mathabx,shuffle,hyperref,tikz,tkz-tab} 
\usepackage[latin1]{inputenc}
\usepackage[T1]{fontenc}   
\usepackage[english]{babel}

\input{xy}
\xyoption{all}

\title{Commutative $B_\infty$-algebras are shuffle algebras}
\date{}
\author{Lo\"\i c Foissy}
\affil{\small{Univ. Littoral C\^ote d'Opale, UR 2597
LMPA, Laboratoire de Math\'ematiques Pures et Appliqu\'ees Joseph Liouville
F-62100 Calais, France}.\\ Email: \texttt{foissy@univ-littoral.fr}}

\author{Fr\'ed\'eric Patras}
\affil{\small{Univ. C\^ote d'Azur et CNRS, UMR 7351 LJAD
F-06108 Nice cedex 2, France}.\\ Email: \texttt{patras@unice.fr}}

\newcommand{\tun}{\begin{tikzpicture}[line cap=round,line join=round,>=triangle 45,x=0.5cm,y=0.5cm]
\clip(-0.2,-0.1) rectangle (0.2,0.2);
\begin{scriptsize}
\draw [fill=black] (0.,0.) circle (1pt);
\end{scriptsize}
\end{tikzpicture}}

\newcommand{\tdeux}{\begin{tikzpicture}[line cap=round,line join=round,>=triangle 45,x=0.5cm,y=0.5cm]
\clip(-.2,-.1) rectangle (0.2,0.7);
\draw [line width=.5pt] (0.,0.5)-- (0.,0.);
\begin{scriptsize}
\draw [fill=black] (0.,0.) circle (1pt);
\draw [fill=black] (0.,0.5) circle (1pt);
\end{scriptsize}
\end{tikzpicture}}

\newcommand{\ttroisun}{\begin{tikzpicture}[line cap=round,line join=round,>=triangle 45,x=0.5cm,y=0.5cm]
\clip(-0.5,-0.1) rectangle (0.5,0.7);
\draw [line width=0.5pt] (0.,0.)-- (-0.3,0.5);
\draw [line width=0.5pt] (0.,0.)-- (0.3,0.5);
\begin{scriptsize}
\draw [fill=black] (-0.3,0.5) circle (1pt);
\draw [fill=black] (0.,0.) circle (1pt);
\draw [fill=black] (0.3,0.5) circle (1pt);
\end{scriptsize}
\end{tikzpicture}}
\newcommand{\ttroisdeux}{\begin{tikzpicture}[line cap=round,line join=round,>=triangle 45,x=0.5cm,y=0.5cm]
\clip(-.2,-.1) rectangle (0.2,1.2);
\draw [line width=0.5pt] (0.,0.5)-- (0.,0.);
\draw [line width=0.5pt] (0.,0.5)-- (0.,1.);
\begin{scriptsize}
\draw [fill=black] (0.,0.) circle (1pt);
\draw [fill=black] (0.,0.5) circle (1pt);
\draw [fill=black] (0.,1.) circle (1pt);
\end{scriptsize}
\end{tikzpicture}}

\newcommand{\tquatreun}{\begin{tikzpicture}[line cap=round,line join=round,>=triangle 45,x=0.5cm,y=0.5cm]
\clip(-0.5,-0.1) rectangle (0.5,0.7);
\draw [line width=0.5pt] (0.,0.)-- (-0.3,0.5);
\draw [line width=0.5pt] (0.,0.)-- (0.3,0.5);
\draw [line width=0.5pt] (0.,0.)-- (0.,0.5);
\begin{scriptsize}
\draw [fill=black] (-0.3,0.5) circle (1.0pt);
\draw [fill=black] (0.,0.) circle (1.0pt);
\draw [fill=black] (0.3,0.5) circle (1.0pt);
\draw [fill=black] (0.,0.5) circle (1.0pt);
\end{scriptsize}
\end{tikzpicture}}
\newcommand{\tquatredeux}{\begin{tikzpicture}[line cap=round,line join=round,>=triangle 45,x=0.5cm,y=0.5cm]
\clip(-0.5,-0.1) rectangle (0.5,1.2);
\draw [line width=0.5pt] (0.,0.)-- (-0.3,0.5);
\draw [line width=0.5pt] (0.,0.)-- (0.3,0.5);
\draw [line width=0.5pt] (-0.3,0.5)-- (-0.3,1.);
\begin{scriptsize}
\draw [fill=black] (-0.3,0.5) circle (1.0pt);
\draw [fill=black] (0.,0.) circle (1.0pt);
\draw [fill=black] (0.3,0.5) circle (1.0pt);
\draw [fill=black] (-0.3,1.) circle (1.0pt);
\end{scriptsize}
\end{tikzpicture}}

\newcommand{\tquatrequatre}{\begin{tikzpicture}[line cap=round,line join=round,>=triangle 45,x=0.5cm,y=0.5cm]
\clip(-0.5,-0.1) rectangle (0.5,1.7);
\draw [line width=0.5pt] (0.,0.)-- (0.,0.5);
\draw [line width=0.5pt] (0.,0.5)-- (0.3,1.);
\draw [line width=0.5pt] (0.,0.5)-- (-0.3,1.);
\begin{scriptsize}
\draw [fill=black] (0.,0.) circle (1.0pt);
\draw [fill=black] (0.,0.5) circle (1.0pt);
\draw [fill=black] (-0.3,1.) circle (1.0pt);
\draw [fill=black] (0.3,1.) circle (1.0pt);
\end{scriptsize}
\end{tikzpicture}}
\newcommand{\tquatrecinq}{\begin{tikzpicture}[line cap=round,line join=round,>=triangle 45,x=0.5cm,y=0.5cm]
\clip(-.2,-.1) rectangle (0.5,1.7);
\draw [line width=0.5pt] (0.,0.)-- (0.,0.5);
\draw [line width=0.5pt] (0.,0.5)-- (0.,1.);
\draw [line width=0.5pt] (0.,1.)-- (0.,1.5);
\begin{scriptsize}
\draw [fill=black] (0.,0.) circle (1.0pt);
\draw [fill=black] (0.,0.5) circle (1.0pt);
\draw [fill=black] (0.,1.) circle (1.0pt);
\draw [fill=black] (0.,1.5) circle (1.0pt);
\end{scriptsize}
\end{tikzpicture}}

\newcommand{\tcinqun}{\begin{tikzpicture}[line cap=round,line join=round,>=triangle 45,x=0.5cm,y=0.5cm]
\clip(-0.7,-0.1) rectangle (0.8,0.7);
\draw [line width=0.5pt] (0.,0.)-- (-0.5,0.5);
\draw [line width=0.5pt] (0.,0.)-- (-0.2,0.5);
\draw [line width=0.5pt] (0.,0.)-- (0.2,0.5);
\draw [line width=0.5pt] (0.,0.)-- (0.5,0.5);
\begin{scriptsize}
\draw [fill=black] (-0.5,0.5) circle (1.0pt);
\draw [fill=black] (-0.2,0.5) circle (1.0pt);
\draw [fill=black] (0.2,0.5) circle (1.0pt);
\draw [fill=black] (0.5,0.5) circle (1.0pt);
\draw [fill=black] (0.,0.) circle (1.0pt);
\end{scriptsize}
\end{tikzpicture}}

\newcommand{\tcinqhuit}{\begin{tikzpicture}[line cap=round,line join=round,>=triangle 45,x=0.5cm,y=0.5cm]
\clip(-0.5,-0.1) rectangle (0.5,1.7);
\draw [line width=0.5pt] (0.,0.)-- (-0.3,0.5);
\draw [line width=0.5pt] (-0.3,0.5)-- (-0.3,1.);
\draw [line width=0.5pt] (0.,0.)-- (0.3,0.5);
\draw [line width=0.5pt] (-0.3,1.)-- (-0.3,1.5);
\begin{scriptsize}
\draw [fill=black] (-0.3,0.5) circle (1.0pt);
\draw [fill=black] (0.,0.) circle (1.0pt);
\draw [fill=black] (-0.3,1.) circle (1.0pt);
\draw [fill=black] (0.3,0.5) circle (1.0pt);
\draw [fill=black] (-0.3,1.5) circle (1.0pt);
\end{scriptsize}
\end{tikzpicture}}

\newcommand{\tcinqtreize}{\begin{tikzpicture}[line cap=round,line join=round,>=triangle 45,x=0.5cm,y=0.5cm]
\clip(-0.5,-0.1) rectangle (0.5,1.7);
\draw [line width=0.5pt] (0.,0.)-- (0.,0.5);
\draw [line width=0.5pt] (0.,0.5)-- (0.,1.);
\draw [line width=0.5pt] (0.,1.)-- (-0.3,1.5);
\draw [line width=0.5pt] (0.,1.)-- (0.3,1.5);
\begin{scriptsize}
\draw [fill=black] (0.,0.) circle (1.0pt);
\draw [fill=black] (0.,0.5) circle (1.0pt);
\draw [fill=black] (0.,1.) circle (1.0pt);
\draw [fill=black] (-0.3,1.5) circle (1.0pt);
\draw [fill=black] (0.3,1.5) circle (1.0pt);
\end{scriptsize}
\end{tikzpicture}}
\newcommand{\tcinqquatorze}{\begin{tikzpicture}[line cap=round,line join=round,>=triangle 45,x=0.5cm,y=0.5cm]
\clip(-.2,-.1) rectangle (0.5,2.2);
\draw [line width=0.5pt] (0.,0.)-- (0.,0.5);
\draw [line width=0.5pt] (0.,0.5)-- (0.,1.);
\draw [line width=0.5pt] (0.,1.)-- (0.,1.5);
\draw [line width=0.5pt] (0.,1.5)-- (0.,2.);
\begin{scriptsize}
\draw [fill=black] (0.,0.) circle (1.0pt);
\draw [fill=black] (0.,0.5) circle (1.0pt);
\draw [fill=black] (0.,1.) circle (1.0pt);
\draw [fill=black] (0.,1.5) circle (1.0pt);
\draw [fill=black] (0.,2.) circle (1.0pt);
\end{scriptsize}
\end{tikzpicture}}

\newcommand{\tdun}[1]{\begin{tikzpicture}[line cap=round,line join=round,>=triangle 45,x=0.5cm,y=0.5cm]
\clip(-.2,-.1) rectangle (0.5,0.5);
\begin{scriptsize}
\draw [fill=black] (0.,0.) circle (1pt);
\end{scriptsize}
\draw(0.3,0.1) node {\tiny #1};
\end{tikzpicture}}

\newcommand{\tddeux}[2]{\begin{tikzpicture}[line cap=round,line join=round,>=triangle 45,x=0.5cm,y=0.5cm]
\clip(-.2,-.1) rectangle (0.5,1.);
\draw [line width=.5pt] (0.,0.5)-- (0.,0.);
\begin{scriptsize}
\draw [fill=black] (0.,0.) circle (1pt);
\draw [fill=black] (0.,0.5) circle (1pt);
\end{scriptsize}
\draw(0.3,0.1) node {\tiny #1};
\draw(0.3,0.6) node {\tiny #2};
\end{tikzpicture}}

\newcommand{\tdtroisun}[3]{\begin{tikzpicture}[line cap=round,line join=round,>=triangle 45,x=0.5cm,y=0.5cm]
\clip(-0.5,-0.1) rectangle (0.5,1.2);
\draw [line width=0.5pt] (0.,0.)-- (-0.3,0.5);
\draw [line width=0.5pt] (0.,0.)-- (0.3,0.5);
\begin{scriptsize}
\draw [fill=black] (-0.3,0.5) circle (1pt);
\draw [fill=black] (0.,0.) circle (1pt);
\draw [fill=black] (0.3,0.5) circle (1pt);
\end{scriptsize}
\draw(0.35,0.1) node {\tiny #1};
\draw(0.3,0.8) node {\tiny #2};
\draw(-0.3,0.8) node {\tiny #3};
\end{tikzpicture}}
\newcommand{\tdtroisdeux}[3]{\begin{tikzpicture}[line cap=round,line join=round,>=triangle 45,x=0.5cm,y=0.5cm]
\clip(-.2,-.1) rectangle (0.5,1.5);
\draw [line width=0.5pt] (0.,0.5)-- (0.,0.);
\draw [line width=0.5pt] (0.,0.5)-- (0.,1.);
\begin{scriptsize}
\draw [fill=black] (0.,0.) circle (1pt);
\draw [fill=black] (0.,0.5) circle (1pt);
\draw [fill=black] (0.,1.) circle (1pt);
\end{scriptsize}
\draw(0.3,0.1) node {\tiny #1};
\draw(0.3,0.6) node {\tiny #2};
\draw(0.3,1.1) node {\tiny #3};
\end{tikzpicture}}

\newcommand{\ptroisun}{\begin{tikzpicture}[line cap=round,line join=round,>=triangle 45,x=0.5cm,y=0.5cm]
\clip(-0.5,-0.1) rectangle (0.5,0.7);
\draw [line width=0.5pt] (0.,0.5)-- (-0.3,0.);
\draw [line width=0.5pt] (0.,0.5)-- (0.3,0.);
\begin{scriptsize}
\draw [fill=black] (-0.3,0.) circle (1pt);
\draw [fill=black] (0.,0.5) circle (1pt);
\draw [fill=black] (0.3,0.) circle (1pt);
\end{scriptsize}
\end{tikzpicture}}

\newcommand{\pquatreun}{\begin{tikzpicture}[line cap=round,line join=round,>=triangle 45,x=0.5cm,y=0.5cm]
\clip(-0.5,-0.1) rectangle (0.5,0.7);
\draw [line width=0.5pt] (0.,0.5)-- (-0.3,0.);
\draw [line width=0.5pt] (0.,0.5)-- (0.3,0.);
\draw [line width=0.5pt] (0.,0.5)-- (0.,0.);
\begin{scriptsize}
\draw [fill=black] (-0.3,0.) circle (1.0pt);
\draw [fill=black] (0.,0.5) circle (1.0pt);
\draw [fill=black] (0.3,0.) circle (1.0pt);
\draw [fill=black] (0.,0.) circle (1.0pt);
\end{scriptsize}
\end{tikzpicture}}
\newcommand{\pquatredeux}{\begin{tikzpicture}[line cap=round,line join=round,>=triangle 45,x=0.5cm,y=0.5cm]
\clip(-0.5,-0.1) rectangle (0.5,1.2);
\draw [line width=0.5pt] (0.,1.)-- (-0.3,0.5);
\draw [line width=0.5pt] (0.,1.)-- (0.3,0.5);
\draw [line width=0.5pt] (-0.3,0.5)-- (-0.3,0.);
\begin{scriptsize}
\draw [fill=black] (-0.3,0.5) circle (1.0pt);
\draw [fill=black] (0.,1.) circle (1.0pt);
\draw [fill=black] (0.3,0.5) circle (1.0pt);
\draw [fill=black] (-0.3,0.) circle (1.0pt);
\end{scriptsize}
\end{tikzpicture}}

\newcommand{\pquatrequatre}{\begin{tikzpicture}[line cap=round,line join=round,>=triangle 45,x=0.5cm,y=0.5cm]
\clip(-0.5,-0.1) rectangle (0.5,1.7);
\draw [line width=0.5pt] (0.,1.)-- (0.,0.5);
\draw [line width=0.5pt] (0.,0.5)-- (0.3,0.);
\draw [line width=0.5pt] (0.,0.5)-- (-0.3,0.);
\begin{scriptsize}
\draw [fill=black] (0.,1.) circle (1.0pt);
\draw [fill=black] (0.,0.5) circle (1.0pt);
\draw [fill=black] (-0.3,0.) circle (1.0pt);
\draw [fill=black] (0.3,0.) circle (1.0pt);
\end{scriptsize}
\end{tikzpicture}}
\newcommand{\pquatrecinq}{\begin{tikzpicture}[line cap=round,line join=round,>=triangle 45,x=0.5cm,y=0.5cm]
\clip(-0.5,-0.1) rectangle (0.5,0.7);
\draw [line width=0.5pt] (-0.3,0.)-- (-0.3,0.5);
\draw [line width=0.5pt] (0.3,0.)-- (0.3,0.5);
\draw [line width=0.5pt] (-0.3,0.)-- (0.3,0.5);
\begin{scriptsize}
\draw [fill=black] (-0.3,0.) circle (1pt);
\draw [fill=black] (0.3,0.) circle (1pt);
\draw [fill=black] (-0.3,0.5) circle (1pt);
\draw [fill=black] (0.3,0.5) circle (1pt);
\end{scriptsize}
\end{tikzpicture}}

\newcommand{\pquatresept}{\begin{tikzpicture}[line cap=round,line join=round,>=triangle 45,x=0.5cm,y=0.5cm]
\clip(-0.5,-0.1) rectangle (0.5,0.7);
\draw [line width=0.5pt] (-0.3,0.)-- (-0.3,0.5);
\draw [line width=0.5pt] (0.3,0.)-- (0.3,0.5);
\draw [line width=0.5pt] (0.3,0.)-- (-0.3,0.5);
\draw [line width=0.5pt] (-0.3,0.)-- (0.3,0.5);
\begin{scriptsize}
\draw [fill=black] (-0.3,0.) circle (1pt);
\draw [fill=black] (0.3,0.) circle (1pt);
\draw [fill=black] (-0.3,0.5) circle (1pt);
\draw [fill=black] (0.3,0.5) circle (1pt);
\end{scriptsize}
\end{tikzpicture}}
\newcommand{\pquatrehuit}{\begin{tikzpicture}[line cap=round,line join=round,>=triangle 45,x=0.5cm,y=0.5cm]
\clip(-0.5,-0.1) rectangle (0.5,1.2);
\draw [line width=0.5pt] (0.,1.)-- (-0.3,0.5);
\draw [line width=0.5pt] (0.,1.)-- (0.3,0.5);
\draw [line width=0.5pt] (-0.3,0.5)-- (0.,0.);
\draw [line width=0.5pt] (0.3,0.5)-- (0.,0.);
\begin{scriptsize}
\draw [fill=black] (-0.3,0.5) circle (1.0pt);
\draw [fill=black] (0.,1.) circle (1.0pt);
\draw [fill=black] (0.3,0.5) circle (1.0pt);
\draw [fill=black] (0.,0.) circle (1.0pt);
\end{scriptsize}
\end{tikzpicture}}

\newcommand{\pdtroisun}[3]{\begin{tikzpicture}[line cap=round,line join=round,>=triangle 45,x=0.5cm,y=0.5cm]
\clip(-0.7,-0.1) rectangle (0.7,1.2);
\draw [line width=0.5pt] (0.,0.5)-- (-0.3,0.);
\draw [line width=0.5pt] (0.,0.5)-- (0.3,0.);
\begin{scriptsize}
\draw [fill=black] (-0.3,0.) circle (1pt);
\draw [fill=black] (0.,0.5) circle (1pt);
\draw [fill=black] (0.3,0.) circle (1pt);
\end{scriptsize}
\draw(0.,0.8) node {\tiny #1};
\draw(-0.6,0.1) node {\tiny #2};
\draw(0.6,0.1) node {\tiny #3};
\end{tikzpicture}}

\theoremstyle{plain}
\newtheorem{theo}{Theorem}[section]
\newtheorem{lemma}[theo]{Lemma}
\newtheorem{cor}[theo]{Corollary}
\newtheorem{prop}[theo]{Proposition}
\newtheorem{defi}[theo]{Definition}

\theoremstyle{remark}
\newtheorem{remark}{Remark}[section]

\newtheorem{example}{Example}[section]

\newcommand{\K}{\mathbb{K}}
\newcommand{\N}{\mathbb{N}}
\newcommand{\id}{\mathrm{Id}}

\newcommand{\sh}{\mathrm{Sh}}
\renewcommand{\ker}{\mathrm{Ker}}
\newcommand{\prim}{\mathit{Prim}}
\newcommand{\desc}{\mathit{Desc}}

\begin{document}
\maketitle
\section{Introduction}
The purpose of the present article is to develop systematically the theory of commutative $B_\infty$-algebras, that is commutative Hopf algebras that are cofree as conilpotent coalgebras. In concrete terms, up to the choice of a basis $B$ of the vector subspace of primitive elements, a commutative $B_\infty$-algebra is a commutative Hopf algebra $H$ that has furthermore  a basis indexed by words (elements of the free monoid $B^\ast$ over $B$) in such a way that the expression of the coproduct in that basis is the deconcatenation coproduct of words. 
Various examples are provided by familiar objects in algebra, combinatorics and topology:  to quote a few, shuffle and quasi-shuffle Hopf algebras over a commutative algebras; Hopf algebra structures constructed on finite topologies, quasi-orders, orders; the Hopf algebra of quasi-symmetric functions; the Hopf algebras dual to enveloping algebras of free Lie algebras...

The choice of such a basis is in general highly non canonical and there is furthermore an infinity of such bases. In this context, a change of $B_\infty$-structure on $H$ can be understood as a change of basis. The most classical example of such phenomena in algebraic combinatorics, the theory of free Lie algebras and the structure theory of Hopf algebras is probably provided by change of bases in the Hopf algebra of quasi-symmetric functions or, dually, by change of bases in the Hopf algebra of descents in symmetric groups (see Example \ref{dsg}). We will actually show in the present article that there is a very close connection between the structure theory of commutative $B_\infty$-algebras and the one of descent algebras that extends the classical connection between the structure theory of Hopf algebras and the one of descent algebras.

 The main purpose of the article will thus be to investigate changes of $B_\infty$ structures on a given Hopf algebra. 
In order to study $B_\infty$ structures, we first construct an adapted categorical framework (the one of filtered-graded coalgebras). This framework accounts for the fact that in spite of commutative $B_\infty$-algebras not carrying a graded Hopf algebra structure in general (more precisely, the existence of a grading is not part or a direct consequence of their definition), the product of a commutative $B_\infty$-algebra has always a nice behavior with respect to the filtration induced by the cofree coalgebra structure. 

We investigate then the case of graded commutative $B_\infty$-algebras $H$, that is graded connected Hopf algebras equipped with the structure of a cofree graded coalgebra. 
Many examples in the literature correspond to this case, that has been investigated recently by C. Bellingeri, E. Ferrucci and N. Tapia\footnote{We thank them warmly for pointing out to us the relevance of their article on branched rough paths for the theory of $B_\infty$-algebras. Our initial project was focusing on the structure of general $B_\infty$-algebras; their article was a key motivation  to extend the scope of the article and systematically study the structure of graded $B_\infty$-algebras.} in the context of the (Butcher)-Connes-Kreimer Hopf algebra $H_{CK}$ of non planar trees and its applications to the theory of branched rough paths. As they pointed out, although their results focus on $H_{CK}$, they would hold more generally for arbitrary functors from the category of vector spaces to graded commutative $B_\infty$-algebras: in particular, such a functor is isomorphic (in the category of commutative Hopf algebras, by a precise isomorphism) to the shuffle Hopf algebra functor \cite[Remark 3.5]{BFT}. 

We recover their results on graded commutative $B_\infty$-algebras and augment their approach in several respects. Firstly, we systematically use the fact that there is, in the (locally finite dimensional) graded case, a perfect duality between a graded commutative Hopf algebras and the dual graded cocommutative Hopf algebra. This duality allows to prove jointly various structure theorems for graded commutative $B_\infty$-algebras and their graded duals --- we emphasize that working on the dual side makes things more standard and easier to grasp: recall Serre's famous wit:  ``there is a general
principle: every calculation relative to coalgebras is trivial and incomprehensible''.
We prove for example that these duals are always, in infinitely many ways that depend on the choice of a generating vector subspace, enveloping algebras of free Lie algebras.
Our arguments rely only on standard theorems in the theory of Lie algebras such as the Poincar\'e-Birkhoff-Witt theorem. 

Secondly, as the (graded) dual of the enveloping algebra of a free Lie algebra is a shuffle Hopf algebra,
we deduce from the duality with enveloping algebras of free Lie algebras that {\it any} family of Lie idempotents (Eulerian, Dynkin, Klyachko... --- there are infinitely many of them) gives rise to a Hopf algebra isomorphism between a graded commutative $B_\infty$-algebra and a shuffle Hopf algebra. The isomorphism introduced in \cite[Remark 3.5]{BFT} corresponds to the particular case of the Eulerian family. Our results imply in particular that the combinatorial theory of graded commutative $B_\infty$-algebras is much richer than what could be expected.

We turn then to the general (non graded) case. It includes in particular quasi-shuffle Hopf algebras over a commutative algebra $A$, studied in detail in \cite{FP20}, that provide a right framework to investigate for example It\^o to Stratonovich transformations for semimartingales \cite{ebrahimi2015flows}. 
We investigate first isomorphisms between a commutative $B_\infty$-algebra $H$ a shuffle Hopf algebra using general Hopf algebra techniques that were first developed to investigate the properties of Hopf algebra endomorphisms. We obtain a parametrisation of such isomorphisms by what we call tangent-to-identity infinitesimal endomorphisms of $H$.
Up to some point, this extends the study that we performed with Jean-Yves Thibon in \cite{foissy2} --- in the language of the present article, that article studied deformations of the shuffle Hopf algebra and the quasi-shuffle Hopf algebra functors from commutative algebras to  commutative $B_\infty$-algebras.
We use then an idea that was implicitly used in several of our earlier works and was developed systematically in \cite{CP21}, namely the fact that classical structure theorems and key properties of graded connected commutative or cocommutative Hopf algebras still hold under the assumption that
the Hopf algebra is unipotent (that is, that its identity map is locally unipotent for the convolution product \cite[Def. 4.1.2]{CP21}).
This allows to adapt to the non graded case the results obtained in the graded case. This adaptation is however partial as the combinatorial structure theory of unipotent Hopf algebras still has to be systematically developed. 

The article is organized as follows. We recall first in Section \ref{cofalg} definitions in relation to coalgebras, introduce immediately the category of filtered-graded coalgebras and prove in Section \ref{cfgc} some key properties of cofree coalgebras in this category. Section \ref{BI} introduces $B_\infty$-algebras and related notions whereas Section \ref{gba} introduces graded $B_\infty$ structures. Section \ref{FE} surveys several fundamental examples. Section \ref{SGB} investigates the structure of graded commutative $B_\infty$-algebras using duality properties, the Poincar\'e-Birkhoff-Witt theorem and the theories of free Lie algebras and Lie idempotents. Section \ref{SCB} studies endomorphisms of commutative $B_\infty$-algebras and  extends the results of Section \ref{SGB} to the non graded case. In the last two sections, we apply these results to the Hopf algebra of finite topologies and proceed to explicit computations, using extra structures associated to a complementary product and coproduct.

\section{Filtered-graded coalgebras}\label{cofalg}

Recall first some definitions and elementary properties of coalgebras. Details and proofs can be found in \cite{CP21}, especially \cite[Sect. 2.13 Graded and Conilpotent Coalgebras]{CP21}.
All vector spaces are defined over a ground field $\K$ of characteristic $0$.

A \it graded vector space \rm is a vector space decomposing as a direct sum
$V=\bigoplus\limits_{n\in\N}V_n$. It is \it reduced \rm if $V_0=0$.  Given $V,W$ two graded vector spaces, a morphism of graded vector spaces from $V$ to $W$ is a morphism of vector spaces $\phi:V\to W$  that respects the graduation ($\phi(V_n)\subset W_n$ for any $n\in \N$).
A graded vector space is locally finite dimensional if all the $V_n$ are finite-dimensional vector spaces.
The graded dual of a graded vector space is the graded vector space $V^\ast=\bigoplus\limits_{n\in\N}V_n^\ast$, where $V_n^\ast$ stands for the dual of $V_n$. It  is locally finite dimensional when $V$ is such.
The tensor product of two graded vector spaces is defined by
$V\otimes W=\bigoplus\limits_{n\in{\N}} (V\otimes W)_n$, where 
$(V\otimes W)_n:= \bigoplus\limits_{p+q=n}V_p\otimes W_q$.
The ground field $\K$ identifies with the graded vector space, still written $\K$, with only one non zero component, $\K_0=\K$. It is the unit of the tensor product ($V\otimes \K=V=V\otimes \K$).

A \it filtered \rm vector space is a vector space $V$ equipped with an increasing filtration by subspaces
$V_0\subset V_1\subset\cdots \subset V_n\subset\cdots$, such that $V=\bigcup\limits_{n\in \N}V_n$.  Given $V,W$ two filtered vector spaces, a morphism of filtered vector spaces from $V$ to $W$ is a morphism of vector spaces $\phi:V\to W$  that respects the filtration ($\phi(V_n)\subset W_n$ for any $n\in \N$).
The tensor product of two filtered vector spaces is defined by requiring $(V\otimes W)_n:=\sum\limits_{i+j=n}V_i\otimes W_j$. Each graded vector space $V$ is canonically filtered: writing $F$ for the functor from graded to filtered vector spaces, $FV_n:=\bigoplus\limits_{p\leq n}V_p$. We will be mostly interested in the present article in Hopf algebras that are filtered as algebras (in the sense that the product map is a map of filtered vector spaces) and graded as coalgebras. 

A graded coalgebra is a coalgebra in the tensor category of graded vector spaces.
That is, graded coalgebras and other graded structures are defined as usual except for the fact that structure morphisms, for example the coproduct $\Delta$ of a graded coalgebra $C$, has to be a morphism of graded vector spaces, so that 
\[\Delta: C_n\to \bigoplus\limits_{p+q=n}C_p\otimes C_q.\]
The counit map from $C$ to $\K$, that we will write $\varepsilon$ (or $\varepsilon_C$ if we want to emphasize what is the underlying coalgebra) is automatically a null map except in degree $0$.
A graded coalgebra $C$ is connected if $C_0=\K$; for such a coalgebra, $\Gamma(C)=\{1\}$, where $\Gamma(C)$ stands for the set of group-like elements in $C$, that is
\[\Gamma(C):=\{c\in C, \:c\not=0 \mbox{ and } \Delta(c)=c\otimes c\}.\]

\begin{defi}
The category $\mathbf{FgCoalg}$ of filtered-graded coalgebras is the category whose objects are graded coalgebras and the set of morphisms between two graded coalgebras $C$ and $D$ the set of filtered vector spaces morphisms of coalgebras from $C$ to $D$.
\end{defi}
That is, a morphism of filtered-graded coalgebras from $C$ to $D$ is a morphism $\phi$ of coalgebras such that $\phi(C_n)\subseteq \bigoplus\limits_{p\leq n}D_p$ for any $n\in \N$. 

Let now $(C,\Delta,\varepsilon)$ be a coaugmented coalgebra with coaugmentation $\eta: \K\to C$ (a coaugmentation is a map of coalgebras from the ground field to $C$, where the ground field is equipped with the identity coproduct: $\K=\K\otimes \K$). In other terms, we fix a group-like element $1\in C$, that is to say a nonzero element $1$ such that $\Delta(1)=1\otimes 1$. The coaugmentation is a section of the projection $\varepsilon$ to the ground field and one has the decomposition $C=\bar C\oplus\K$, where $\bar C:= \ker(\varepsilon)$.
Notice that a graded connected coalgebra is canonically coaugmented: its coaugmentation is the isomorphism between the ground field and the degree 0 component of the coalgebra, and $1$ is the unique group-like element of $C$. In general, for a coaugmented coalgebra one can define the reduced coproduct from $\bar C$ to $\bar C\otimes\bar C$ by
\[\overline \Delta(c):=\Delta(c)-c\otimes 1-1\otimes c;\]
the element $c$ is called primitive if $\overline\Delta (c)=0$, the vector space of primitive elements is denoted $\prim(C)$.

The iterated reduced coproduct from $\bar C$ to $\bar C^{\otimes n+2}$ is inductively defined by 
\begin{align*}
\bar\Delta_{n+2}&:=(\bar\Delta\otimes Id_{\bar C}^{\otimes n})\circ\bar\Delta_{n+1},\mbox{ with }\bar\Delta_2:=\bar\Delta.
\end{align*}
Let $F_n:= \ker\ (\overline \Delta_{n+1})\subset\bar C$. Notice that $\bar\Delta_{n+2}=(\bar\Delta\otimes Id_{\bar C}^{\otimes n})\circ\bar\Delta_{n+1}$ implies $F_n\subset F_{n+1}$.

\begin{lemma}\label{lemmeconil}
The coproduct $\Delta$ and the coalgebra $C$ are called conilpotent if $\bar C=\bigcup\limits_{n\geq 1}  F_n$, that is if for every $c\in \bar C$ there exists an integer $n\geq 2$ such that $\overline \Delta_n(c)=0$. When the coalgebra $C$ is graded and connected, the coproduct is automatically conilpotent.
\end{lemma}

\begin{proof}
 Indeed, when $C$ is graded connected, $\bar C_0=0$. As $\bar\Delta_k:\bar C_n\to\bigoplus\limits_{\substack{i_1+\dots+i_k=n,\\ i_1,\ldots,i_k\geq 1}}\bar C_{i_1}\otimes\dots\bar C_{i_k}$, the iterated coproduct $\overline \Delta_k$, $k\geq 2$, vanishes on $C_{n}$ for $1\leq n\leq k-1$.
 \end{proof}
 
\section{Cofree filtered-graded coalgebras}\label{cfgc}

Let now $V$ be a vector space over $\K$. 
We denote by $T(V):=\bigoplus\limits_{n=0}^\infty V^{\otimes n}$ the tensor gebra \footnote{We use the terminology of \cite{CP21} and call gebra a vector space that can be equipped with several algebraic structures --- this allows in particular to avoid calling ``tensor algebra'' the vector space $T(V)$ without equipping it with the algebra structure obtained from the concatenation product of words.}of $V$, and use the word notation for tensors (that is, $v_1\ldots v_n$ will stand for $v_1\otimes\cdots\otimes v_n$). We will use later a similar notation for the tensor gebra over an alphabet $X$ (the linear span of $X^\ast$, the set of words --- or free monoid --- over $X$).
The tensor gebra is graded: the degree of a word $w$ is its length $\ell(w)$, that is, the number of its letters. 

We equip the tensor gebra with the deconcatenation coproduct, that makes it, together with the canonical projection $\varepsilon$ to $\K=V^{\otimes 0}$, a graded (counital, conilpotent, connected, coassociative) coalgebra:
for any $v_1,\ldots,v_n\in V$, with $n\geq 0$,
\[\Delta(v_1\ldots v_n)=\sum_{i=0}^n v_1\ldots v_i\otimes  v_{i+1}\ldots v_n.\]

We write $T_+(V):=\bigoplus\limits_{n=1}^\infty V^{\otimes n}$ for $\overline{T(V)}$. As mentioned in Section \ref{cofalg}, it is equipped with the (coassociative but not counital) reduced deconcatenation coproduct:
for any $v_1,\ldots,v_n\in V$, with $n\geq 1$,
\[\overline\Delta(v_1\ldots v_n)=\sum_{i=1}^{n-1} v_1\ldots v_i\otimes  v_{i+1}\ldots v_n.\]

The following Lemma justifies the introduction of the category of filtered-graded coalgebras.
\begin{lemma}\label{filtcoalg}
Let $C$ be graded connected coalgebra and $\phi$ be a morphism of coalgebras from $C$ to $(T(V),\Delta)$. Then, $\phi$ is a morphism of filtered graded coalgebras. That is, $\phi(C_n)\subseteq FT(V)_n=\bigoplus\limits_{k=0}^nV^{\otimes k}$.
\end{lemma}
\begin{proof}
As $\phi$ is a coalgebra morphism between connected coalgebras, $(\phi\otimes \phi)\circ \overline\Delta=\overline\Delta\circ \phi$, this implies that for any $n\in \N$,
\[\phi^{\otimes n}\circ\overline\Delta_n=\overline\Delta_n\circ\phi.\]
Therefore, $\phi(\ker(\overline\Delta_{n+1}))\subseteq \ker(\overline\Delta_{n+1})$. It is an easy exercise to show that in $T(V)$, $\ker(\overline\Delta_{n+1})=FT(V)_n$. 
Moreover, as we have already seen, it holds that $C_n\subseteq \ker(\overline\Delta_{n+1})$. We get finally $\phi(C_n)\subseteq FT(V)_n$. 
\end{proof}

\begin{defi}[Cofree filtered-graded coalgebras]\label{cofreedef}
Let $V$ be a reduced graded vector space. A cofree filtered-graded coalgebra over $V$ is a connected graded coalgebra $C$ together with a filtered vector space map $\pi$ from $\bar C$ to $V$
 such that
for each connected graded coalgebra $D$, any filtered vector space morphism $\phi$ from $\bar D$ to $V$ lifts uniquely to a morphism $\Phi$ of coalgebras from $D$ to $C$ in $\mathbf{FgCoalg}$, such that $\pi\circ \Phi_{\mid \bar D}=\phi$. A cofree filtered-graded coalgebra is called reduced if $V=V_1$. The space $V$ is called the cofreely cogenerating space of $C$; the map $\pi$ is called the structure map.
\end{defi}
As usual for cofree objects, any two cofree filtered-graded coalgebras over $V$ are isomorphic (by a unique isomorphism, see Lemma \ref{uniso} below). This justifies to call (slightly abusively) any cofree filtered-graded coalgebra over $V$, ``the'' cofree filtered-graded coalgebra over $V$.

\begin{remark} Cofree graded coalgebras are defined similarly: just replace everywhere ``filtered-graded'' and ``filtered'' by ``graded'' in the Definition.\end{remark}

\begin{lemma}\label{cofreeuniv} The tensor gebra $T(V)$ over a graded vector space $V$, equipped with the deconcatenation coproduct $\Delta$, the canonical projection $\pi_V$ from $T_+(V)$ to $V$ and the graduation $|v_1\ldots v_n|:=|v_1|+\ldots +|v_n|$ for $v_i\in V,\ i=1\dots n$, is a cofree filtered-graded coalgebra over $V$ denoted $T^g(V)$ (the exponent $g$ is introduced to indicate that the graduation is not the tensor degree one excepted in the reduced case, that is when $V=V_1$). We call it the standard cofree filtered-graded coalgebra over $V$. With our previous notation, the morphism $\Phi$ from $D$ to $T^g(V)$ is obtained as
\[\varepsilon+\sum\limits_{n\geq 1}\phi^{\otimes n}\circ\bar\Delta_n.\]
\end{lemma}
In the formula, it is understood that if $d\in D$ decomposes as $\varepsilon (d)+(d-\varepsilon (d))\in\mathbb K\oplus \bar D$, 
\[\Phi(d)=\varepsilon(d)+\sum\limits_{n\geq 1}\phi^{\otimes n}\circ\bar\Delta_n(d-\varepsilon (d)).\]
\begin{proof}
Recall that $(T(V),V,\pi_V)$ is a cofree conilpotent coalgebra, that the formula for $\Phi$ holds in the cofree conilpotent case (see e.g. \cite[Exercise 2.13.3, Remark 2.13.1]{CP21}) and that graded connected coalgebras are conilpotent. The fact that $\Phi$ is a filtered-graded morphism of coalgebras follows from the fact that the coproduct is a graded map from $T^g(V)$ to $T^g(V)\otimes T^g(V)$ and that $\phi$ is a filtered map from $\overline D$ to $V$. This concludes the proof.\end{proof}

Leaving aside graduations, the simplest way to understand and prove the Lemma and the formula for $\Phi$ is by dualizing the statement and using the fact that the tensor gebra equipped with the concatenation product of words is a free associative algebra. Given indeed a map $\phi^\ast$ from the dual of $V$, $V^\ast$, to the projective limit of vector spaces $\cdots\to FD_{n+1}^\ast \to FD_n^\ast\to \cdots\to FD_1^\ast$, where $FD_n^\ast=D_1^\ast\oplus\cdots\oplus D_n^\ast$ and $D^\ast$ is a graded connected algebra, this map uniquely extends to an algebra map from $T(V^\ast)$ to $\hat D=\K\oplus \prod\limits_{n\in\N^\ast}D_n^\ast$, the completion of the graded dual of $D$ with respect to its canonical filtration. This algebra map $\xi$ is given as usual on $T_+(V^\ast)$ by 
\[\xi(v_1^\ast\cdots v_n^\ast):=\phi^\ast(v_1^\ast)\cdots \phi^\ast(v_n^\ast),\]
which is dual to the formula for $\Phi$ in the Lemma. The fact that $D^\ast$ is a graded algebra insures that $\xi$ is well-defined as the degree $p$ component of $\phi^\ast(v_1)\cdots \phi^\ast(v_n)$ is obtained as the sum of a finite number of terms, for any $p$.

\begin{remark}
Notice that since $V=\prim(T^g(V))$, the vector space of primitive elements of $T^g(V)$, it holds for any cofree filtered-graded (resp. cofree graded, cofree conilpotent) coalgebra $C$ over $V$ that $\prim(C)\cong V$. In particular one can always choose $V$ to be $\prim(C)$.
Fix now a graded basis $(b_i)_{i\in J}$ of $\prim(C)$. The choice of a structure map $\pi$ can be interpreted as the choice of a basis of $C$ since it induces an isomorphism 
$C\cong T^g(\prim(C))$, where $T^g(\prim(C))$ has the basis $b_{j_1}\ldots b_{j_k},\ j_l\in J, l=1,\ldots ,k,\ k\in\mathbb N^\ast$. We will say that the choice of a cofree filtered-graded (resp. cofree graded, cofree conilpotent) coalgebra structure $\pi:C\to \prim(V)$ determines a {\it presentation} of $C$ and call the associated basis  (resp. graded basis, resp. basis) of $C$ the $\pi$-basis.
\end{remark}
\begin{remark}
The isomorphism $C\cong T^g(\prim(C))$ is non canonical as it depends heavily on the choice of $\pi$. This observation is made more precise in the following Lemma and its proof.
\end{remark}
\begin{lemma}\label{uniso}\label{unisograd} Let $C$ be a connected graded coalgebra. Cofree filtered graded (resp. cofree graded) coalgebra structures on $C$, whenever they exist, are parametrized by the surjections from $\overline C$ to $\prim(C)$ that are maps of filtered (resp. graded) vector spaces and restrict to the identity map on $\prim(C)$. 
All such structures are isomorphic:
a cofree filtered-graded (resp. graded) coalgebra structure on a connected graded coalgebra $C$ is thus unique up to isomorphisms.
\end{lemma}
\begin{proof}
Fixing a cofree filtered-graded coalgebra structure on $C$ amounts to fixing an isomorphism of coalgebras between $C$ and the standard cofree filtered-graded coalgebra over $V:=\prim(C)$. We can therefore assume without restriction that $C$ is $T^g(V)$, the standard cofree filtered-graded coalgebra over $V$ with structure map denoted $\pi_V$ (the canonical projection from $T^g(V)$ to $V$). Any other cofree filtered graded coalgebra structure on $C$ is given by a filtered surjection $\pi$ from $C$ to $V$ that restricts to the identity map on $V$. The map $\pi$ is the structure map of the second cofree structure; by Definition \ref{cofreedef} it induces a filtered-graded coalgebra endomorphism $\tilde\pi$ of $T^g(V)$ obtained as:
\begin{equation}\label{unifla}
v_1\ldots v_n\longmapsto \sum\limits_{k<n}\sum\limits_{w_1\ldots w_k=v_1\ldots v_n}\pi(w_1)\ldots \pi(w_k).
\end{equation}
It can be inverted, and the inverse isomorphism $\tilde\mu$ such that $\tilde\mu\circ\tilde\pi =Id_{T^g(V)}$ is entirely characterized by the identity 
$$\pi_V\circ\tilde\mu\circ\tilde\pi=\pi_V,$$
that is, $\mu(v)=v$ for $v\in V$ and, for $k\geq 2$ and $v_1,\ldots,v_n\in V$),
\begin{equation}\label{inversestru}
\mu(v_1\ldots v_n)=-\sum\limits_{k<n}\sum\limits_{w_1\ldots w_k=v_1\ldots v_n}\mu(\pi(w_1)\ldots\pi(w_k)).
\end{equation}
The same arguments and formulas apply in the graded case.
\end{proof}
\begin{remark}
Unicity up to isomorphism of free and cofree objects is a universal phenomena and actually part of the abstract definition of freeness and cofreeness (usually as left and right adjunct functors). Unicity up to isomorphism also holds in particular in the conilpotent case; the formula for the inverse isomorphism is the same as in the filtered-graded and graded cases as it depends only on the formula for $\Phi$.
\end{remark}
\begin{remark} The category of filtered-graded coalgebras is the natural framework to study Hopf algebras such as quasi-shuffle Hopf algebras. See for example the discussion of natural endomorphisms of quasi-shuffle Hopf algebras in \cite{novelli2} for insights on the role of filtered-graded maps in that context (the surjections that appear in that article have indeed to be understood as filtered maps: they map a tensor of order $n$ to a tensor of lower or equal degree).
\end{remark}

\section{$B_\infty$-algebras}\label{BI}
The notion of $B_\infty$-algebra was first introduced by Getzler and Jones in \cite{GJ94} for cochain complexes. Their definition extends to other tensor categories, we consider here $B_\infty$-algebras in the category of vector spaces, see also \cite{F17} for further insights and applications. We survey here the fundamental definitions and properties, using systematically the properties of cofree coalgebras.

Recall, for completeness sake, that a bialgebra $B$ is a unital algebra and a counital coalgebra such that the product and the unit map are morphisms of counital coalgebras. When the coalgebra structure is conilpotent, as it is always the case in the present article, the notions of bialgebra and Hopf algebra identify and we will use the two terminologies indifferently. Denoting $\pi$ and $\Delta$ the product and the coproduct of $B$, the convolution $f\ast g$ of two linear endomorphisms $f,g$ of $B$ is defined by $f\ast g:=\pi\circ (f\otimes g)\circ \Delta$. The convolution product defines a unital algebra structure on the vector space of linear endomorphisms of $B$ with unit the composition $\eta\circ\varepsilon$ of the unit and counit maps. See \cite{CP21} for a systematic treatment.

\begin{defi}
Let $(H,*,\Delta)$ be a conilpotent Hopf algebra. A cofree structure on $H$ is the data of a structure map $\pi:\bar H\to \prim(H)$ making the triple $(H,\Delta,\pi)$ a cofree conilpotent coalgebra over $\prim(H)$. We will say that $(H,*,\Delta,\pi)$ (or simply $(H,\pi)$ when the underlying Hopf algebra structure is obvious) is a cofree Hopf algebra.
\end{defi}
 
A $B_\infty$-algebra structure on a vector space $V$ is equivalent to the definition of a Hopf algebra structure on the cofree conilpotent coalgebra $(T(V),\Delta,\pi_V)$ (see Proposition \ref{equivdef} below). The product map from $T(V ) \otimes T(V )$ to $T(V )$
is a map of conilpotent coalgebras and is entirely characterized by its projection to the subspace $V$. This observation and the application of Lemma \ref{cofreeuniv} lead to the Lemma \ref{lemmaBinf} below.
\begin{remark}
The tensor product of two graded connected coalgebras is a graded connected coalgebra and $T(V)\otimes T(V)$ is thus a connected graded coalgebra. By Lemma \ref{filtcoalg}, the product $\pi$, which is a map of coalgebras, necessarily respects the filtrations: $\pi(T_n(V)\otimes T_m(V))\subset \bigoplus_{k\leq n+m}T_{k}(V)$. In particular, structure theorems obtained in the previous section for coalgebras in $\mathbf{FgCoalg}$ apply to $B_\infty$-algebras.
\end{remark}

\begin{lemma}\label{lemmaBinf}
Let $*:T(V)\otimes T(V)\longrightarrow T(V)$ be a coalgebra map.
It is entirely characterized by the map $\langle-,-\rangle:T(V)\otimes T(V)\longrightarrow V$ defined by 
\[\langle v_1\ldots v_k,v_{k+1}\ldots v_{k+l}\rangle=\pi_V(v_1\ldots v_k*v_{k+1}\ldots v_{k+l}),\]
where $\pi_V:T(V)\longrightarrow V$ is the canonical surjection. 

Then, $1\ast 1=1$ and for any words $w,w'\in T_+(V)$,
\begin{align}
\label{eq1} w*w'&=\sum_{k=1}^\infty \sum_{\substack{w=w_1\ldots w_k,\\ w'=w'_1\ldots w'_k}} \langle w_1,w'_1\rangle\ldots \langle w_k,w'_k\rangle.
\end{align}
Note that in the sum, the words $w_i$ or $w'_j$ can be empty (in what case $w_i$ and $w'_j$ stand for 1 in the terms $\langle w_i,w'_i\rangle$ or $\langle w_j,w'_j\rangle$).

Moreover, 
\begin{align}
\label{eq1bis} w*1&=\sum_{k=1}^\infty \sum_{\substack{w=w_1\ldots w_k}} \langle w_1,1\rangle\ldots \langle w_k,1\rangle.\\
\label{eq1ter} 1*w'&=\sum_{k=1}^\infty \sum_{\substack{w'=w'_1\ldots w'_k}} \langle 1,w'_1\rangle\ldots \langle 1,w'_k\rangle.
\end{align}
\end{lemma}

\begin{proof}As $\ast$ is a coalgebra map, it maps $1\otimes 1$, the unique group-like element in $T(V)\otimes T(V)$, to $1$, the unique group-like element in $T(V)$.
The Lemma follows then from Lemma \ref{cofreeuniv}: $\ast$ is entirely characterized by $\langle-,-\rangle:=\pi_V\circ\ast$, and applying the formula expressing $\ast$ in terms of $\langle-,-\rangle$ yields to Eqs (\ref{eq1},\ref{eq1bis},\ref{eq1ter}).
\end{proof}

\begin{lemma}\label{lemmaunit}
With the same notation, the product $\ast$ is unital, with unit 1 if, and only if, \[\langle-,1\rangle=\langle 1,-\rangle=\pi_V.\]
\end{lemma}
\begin{proof}
The assertion follows directly from the definitions of $\pi_V$, $\langle-,-\rangle$ and Eqs (\ref{eq1bis},\ref{eq1ter}). 
\end{proof}

We assume from now on that the product $\ast$ is unital, with unit $1$.

\begin{lemma}\label{lemmaassoc} With the same notation,
\begin{enumerate}
\item the product $*$ is associative if, and only if, for any $w,w',w''\in T_+(V)$, 
\[\langle w,w'*w''\rangle=\langle w*w',w''\rangle.\]
\item It is commutative if, and only if, for any $w,w'\in T_+(V)$,
\[\langle w,w'\rangle=\langle w',w\rangle.\]
\end{enumerate}
\end{lemma}
\begin{proof}
As the product is a map of coalgebra, so are the maps $\ast\circ(Id\otimes\ast)$ and $\ast\circ(\ast\otimes Id)$ from $T(V)^{\otimes 3}$ to $T(V)$. They are therefore entirely characterized by the composition with $\pi_V$ and associativity follows from \[\pi_V\circ ( \ast\circ(Id\otimes\ast))=\pi_V\circ (\ast\circ(\ast\otimes Id)).\]
The second assertion is proved similarly, noticing that the twist map $w\otimes w'\to w'\otimes w$ is a morphism of coalgebras.
\end{proof}
\begin{defi}
A $B_\infty$-structure on $V$ is a map $\langle-,-\rangle:T(V)\otimes T(V)\longrightarrow V$, such that:
\begin{itemize}
\item For any word $v_1\ldots v_n\in T(V)$, 
\[\langle 1,v_1\ldots v_n\rangle=\langle v_1\ldots v_n,1\rangle=\begin{cases}
v_1\mbox{ if }n=1,\\
0\mbox{ otherwise}.
\end{cases}\]
\item For any words $w,w',w''\in T_+(V)$,
\[\langle w,w'*w''\rangle=\langle w*w',w''\rangle,\]
where $*$ is defined by
\begin{align}
\label{eq2} w*w'&=\sum_{k=1}^\infty \sum_{\substack{w=w_1\ldots w_k,\\ w'=w'_1\ldots w'_k}} \langle w_1,w'_1\rangle\ldots \langle w_k,w'_k\rangle.
\end{align}
\end{itemize}
We shall say that $\langle-,-\rangle$ is commutative if for any $w,w'\in T_+(V)$, $\langle w,w'\rangle=\langle w',w\rangle$.
We shall say that $\langle-,-\rangle$ is trivial if for furthermore any $w,w'\in T_+(V)$, $\langle w,w'\rangle=0$.
\end{defi}

\begin{remark}\label{rkrew} Equation
(\ref{eq2}) can be rewritten in this way: for any $v_1,\ldots,v_{k+l}\in V$,
\begin{align*}
v_1\ldots v_k*v_{k+1}\ldots v_{k+l}&=\sum_{n=1}^{k+l}
\sum_{\substack{\sigma:[k+l]\twoheadrightarrow [n], \\ \sigma(1)\leq \ldots \leq \sigma(k),\\ \sigma(k+1)\leq \ldots \leq \sigma(k+l)}}
\langle v_{\sigma^{-1}(1)}\rangle\ldots \langle v_{\sigma^{-1}(n)}\rangle,
\end{align*}
with the following notation: if $I=\{i_1,\ldots,i_q\}\subseteq [n]$, with $i_1<\ldots<i_p\leq k< i_{p+1}<\ldots <i_q$,
\[v_I=v_{i_1}\ldots v_{i_p}\otimes v_{i_{p+1}}\ldots v_{i_q}.\]
In particular we get the Lemma:
\end{remark}
\begin{lemma}\label{trivial}
The $B_\infty$-algebra structure is trivial if and only if the product $\ast$ is the shuffle product\footnote{Shuffle and quasi-shuffle products can be commutative (as usually the case in Lie theory) or not (as usually in classical algebraic topology where shuffle products appear in relation to cartesian products of simplices). In the present article they will be always commutative, excepted in the example of quasi-shuffle Hopf algebras over an associative algebra where the quasi-shuffle product is noncommutative when the algebra is noncommutative.}, $*=\shuffle$, with
\begin{align*}
v_1\ldots v_k*v_{k+1}\ldots v_{k+l}&=
\sum_{\substack{\sigma:[k+l]\twoheadrightarrow [k+l], \\ \sigma(1)< \ldots < \sigma(k),\\ \sigma(k+1)< \ldots < \sigma(k+l)}}
 v_{\sigma^{-1}(1)}\ldots v_{\sigma^{-1}(n)}.
\end{align*}
See Example \ref{shuffle} for another classical definition of the shuffle product.
In that case, the Hopf algebra $(T(V),\shuffle,\Delta)$ is called the shuffle Hopf algebra over $V$.
\end{lemma}

\begin{defi}
A  $B_\infty$-algebra  (resp. commutative) is a vector space $V$ equipped with a $B_\infty$-structure (resp. commutative). 
\end{defi}

\begin{prop}\label{equivdef}
Let $P(V)$ be the set of products $*$ on $T(V)$ making $(T(V),*,\Delta)$ a bialgebra and by $B_\infty(V)$ the set of $B_\infty$-algebra structures on $V$.
The following map is a bijection:
\[\Theta:\left\{\begin{array}{rcl}
P(V)&\longrightarrow&B_\infty(V)\\
*&\longmapsto&\pi_V\circ *.
\end{array}\right.\]
Denoting by $P^c(V)$ the set of commutative products on $T(V)$ making $(T(V),*,\Delta)$ a bialgebra and by $B_\infty^c(V)$ of commutative $B_\infty$-algebra structures
 on $V$, $\Theta$ induces a bijection from $P^c(V)$ to $B_\infty^c(V)$.
\end{prop}

\begin{proof} As $1$ is the unique group-like of $T(V)$, it is necessarily the unit for the product $\ast$.
By Lemmas \ref{lemmaBinf}, \ref{lemmaunit} and \ref{lemmaassoc}, $\Theta$ is well-defined. If $*\in P(V)$, then it is a coalgebra morphism from $T(V)\otimes T(V)$ to $T(V)$.
 By Lemma \ref{cofreeuniv}, $\Theta$ is injective.
 
 Conversely,
let $\langle-,-\rangle$ in $B_\infty(V)$. The product $*$ associated to it by (\ref{eq1}) is a coalgebra map. It is associative by Lemma \ref{lemmaassoc}
and has $1$ for a unit: $*\in P(V)$, and $\Theta(*)=\langle-,-\rangle$. Thus, $\Theta$ is a bijection.

By the last item of Lemma \ref{lemmaassoc}, $\Theta(P^c(V))=B_\infty^c(V)$. 
\end{proof}

\begin{defi}\label{envbinf}
The bialgebra $(T(V),*,\Delta)$ associated to a $B_\infty$-algebra structure $\langle\ ,\ \rangle$ on $V$ is called the $B_\infty$-enveloping algebra of $(V,\langle\ ,\ \rangle)$. For simplicity, we will abusively also say later on that $(T(V),*,\Delta)$ is a $B_\infty$-algebra, identifying implicitly the data of a $B_\infty$-algebra structure on $V$ with the data of a $B_\infty$-enveloping algebra structure on $T(V)$.
\end{defi}

\section{Graded $B_\infty$-algebras}\label{gba}

Let us consider now the notion of $B_\infty$-algebra in the category of graded vector spaces. We shall see later that many classical examples of $B_\infty$-algebras have such a structure.
We assume therefore in this section that $V$ is a graded vector space. For brevity we do not repeat all definitions in the previous section: they have to be adapted as follows
\begin{itemize}
\item replacing everywhere $T(V)$ by its graded version $T^g(V)$,
\item requiring that {\it all maps} be maps of graded vector spaces,
\item requiring in particular that the product $\ast$ resp. the structure map $\langle\ ,\  \rangle$ be maps of graded coalgebras resp. of graded vector spaces (and {\it not} of filtered-graded coalgebras resp. filtered vector spaces!).
\end{itemize}
For example, the definition of cofree Hopf algebras reads in the graded case:
\begin{defi}
Let $(H,*,\Delta)$ be a connected graded Hopf algebra. A cofree graded structure on $H$ is the data of a graded structure map $\pi:H\to \prim(H)$ making the coalgebra $(H,\Delta)$ a cofree graded coalgebra over $\prim(H)$. We will say that $(H,*,\Delta,\pi)$ (or simply $(H,\pi)$ when the underlying Hopf algebra structure is obvious) is a cofree graded Hopf algebra.
\end{defi}

Using the point of view of Definition \ref{envbinf},
let  $H=(T^g(V),*,\Delta)$ be the $B_\infty$ enveloping algebra of $(V,\langle\ ,\  \rangle)$ in the category of graded vector spaces. It is equipped with a graded connected Hopf algebra structure and we say that $T^g(V)$ is then a graded $B_\infty$-enveloping algebra. For simplicity, we will abusively also say that $(T^g(V),*,\Delta)$ is a graded $B_\infty$-algebra. 

\begin{defi}
When the $B_\infty$ structure is trivial, that is when $\langle w,w'\rangle =0$ for words both of length greater or equal 1, the product $\ast$ is the shuffle product $\shuffle$, $H=(T^g(V),\shuffle,\Delta)$ is a graded Hopf algebra, and we say that it is a standard graded $B_\infty$-algebra.
\end{defi}

\begin{remark} A direct inspection of the formulas defining $B_\infty$-enveloping algebras shows that $(T(V),*,\Delta)$ is a graded Hopf algebra for the tensor degree on $T(V)$ if and only if the $B_\infty$-algebra structure is trivial (that is, $<w,w' >=0$ for $w,w'\in T_+(V)$ and the product is the shuffle product). So, if $V=V_1$, there is a unique graded $B_\infty$-algebra structure on $V$: the trivial one.
\end{remark}

Let now $H=(T^g(V),*,\Delta)$ be a locally finite graded commutative $B_\infty$-enveloping algebra (where locally finite means that the graded components $H_n$ are finite dimensional). The graded dual cocommutative Hopf algebra $H^\ast=\bigoplus\limits_{n\in\mathbb N}H_n^\ast$ is a free associative algebra generated by $\widehat{V}^\ast:= \bigoplus\limits_{n\geq 1}V_n^\ast$. That is, up to a canonical isomorphism, $H^\ast=T(\widehat{V}^\ast)$, the tensor algebra over $\widehat{V}^\ast$ equipped with the concatenation product. 
More generally, if $(H,*,\Delta,\pi)$ is a locally finite cofree graded commutative Hopf algebra, $H^\ast\cong T(\widehat{\prim(H)}^*)$.
The following definition characterizes such Hopf algebras.

\begin{defi}
A localy finite graded connected cocommutative Hopf algebra $H$ freely generated as an associative algebra by a graded subspace $W$ is called  a free Lie-type Hopf algebra. When $W\subseteq\prim(H)$ (that is, when $H$ is primitively generated by $W$), we say that $H$ is a standard free-Lie type Hopf algebra. 
\end{defi}
The reasons for this terminological choice will become clear later on: we will show that such a Hopf algebra is always canonically isomorphic to the enveloping algebra of a free Lie algebra.

Recall already that when $H$ is  a standard free-Lie type Hopf algebra, by standard results in the theory of free Lie algebras it is automatically canonically isomorphic to the enveloping algebra of the free Lie algebra generated by $W$. In that case, the graded dual Hopf algebra is (up to a canonical isomorphism) the shuffle Hopf algebra  over $\hat W^\ast$, so that standard free Lie-type Hopf algebras and standard graded $B_\infty$-algebras are in duality. See \cite{Reutenauer} for details on the duality between enveloping algebras of free Lie algebras and shuffle Hopf algebras.

\section{Fundamental examples}\label{FE}

\begin{example}[Free Lie algebras]\label{flalg}
Let us start with one of the simplest possible non trivial example: the free graded associative algebra $A=\mathbb Q\langle x,y\rangle$ on two generators, $x$ of degree 1 and $y$ of degree 2 (so that, for example, the word $xyxy^2$ is of degree 8).
It is the enveloping algebra of the free Lie algebra over $x$ and $y$ (using the rewriting trick $[a,b]=ab-ba$ to expand iterated commutators in the free Lie algebra into sums of words). It is then natural to equip $A$ with a standard free-Lie type Hopf algebra structure by requiring $x$ and $y$ to be primitive elements (as $A$ is a free associative algebra, this choice entirely determines the Hopf algebra structure on $A$).

Consider now the graded dual of $A$, denoted $A^\ast$ and write $x^\ast$ and $y^\ast$ for the elements dual to $x$ and $y$ in the basis of words. In general, if $y_1\ldots y_n$ is a word in the letters $x$ and $y$ we will write $y_1^\ast\ldots y_n^\ast$ for the corresponding element in the dual in the basis of words over $x$ and $y$.
Writing $V$ for the linear span of $x^\ast$ and $y^\ast$ with the graduation $|x^\ast|=1,\ |y^\ast|=2$, $A^\ast$ identifies to $T^g(V)$ equipped with the deconcatenation product. That is, $A^\ast$ identifies with the cofree graded coalgebra $T^g(V)$ over $V$ with structure map $\pi_V$ mapping $y_1^\ast\ldots y_n^\ast$ to $ 0$ if $n\geq 2$ and mapping $y_1^*\to y_1^*$, where $y_i\in\{x,y\},\ i=1\ldots n$. 

Dualizing the coproduct of $A$ to obtain a product $\times$, $T^g(V)$ becomes a standard graded $B_\infty$-algebra and identifies as a Hopf algebra to the shuffle Hopf algebra over $x^\ast$ and $y^\ast$ (see Remark \ref{rkrew} and Example \ref{shuffle}). In particular, 
\[x^\ast \times x^\ast =x^\ast\shuffle x^\ast= 2 x^\ast x^\ast.\]
See \cite{Reutenauer} for details on these constructions and definitions. The corresponding $B_\infty$-algebra structure on $V$, obtained through the projection $\pi_V$, is trivial.

Consider now the change of variables $t:=x,\ z:=y+x^2$ (it is convenient to notationally distinguish $x$ and $t$). One can rewrite $A=\mathbb Q\langle t,z\rangle$ but $A$ is not a standard free Lie-type Hopf algebra any more with respect to the linear span of $t$ and $z$, as $z$ is not primitive: $\overline \Delta(z)=2 t\otimes t$. Let us use the same notation as above: if $y_1\ldots y_n$ is a word in the letters $t$ and $z$ we will write $y_1^\ast\ldots y_n^\ast$ for the corresponding element in the dual basis to the basis of words over $t$ and $z$. For example, as
\begin{align*}
t&=x,&z&=y+x^2,&zt&=yx+x^3,\\
&&t^2&=x^2,&tz&=xy+x^3,\\
&&&&t^3&=x^3,
\end{align*}
we obtain that
\begin{align*}
x^\ast&=t^\ast,&y^\ast&=z^\ast,&y^\ast x^\ast&=z^\ast t^\ast,\\
&&x^\ast x^\ast&=t^\ast t^\ast+z^\ast,&x^\ast y^\ast&=t^\ast z^\ast,\\
&&&&x^\ast x^\ast x^\ast&=t^\ast t^\ast t^\ast+t^\ast z^\ast+z^\ast t^\ast.
\end{align*}

Take care that with this notation $x^\ast=t^\ast$ and $y^\ast=z^\ast$ but $(t^\ast)^n\not= (x^\ast)^n$ in general. 
The coalgebra $A^\ast$ identifies now to $T(W)$, where $W$ is the linear span of $z^\ast$ and $t^\ast$, equipped with the deconcatenation product (the cofree coalgebra over $W$ with structure map $\pi_W$, mapping $y_1^\ast\ldots y_n^\ast$ to $ 0$ if $n\geq 2$ and maps $y_1^\ast\to y_1^\ast$, where $y_1,\ldots,y_n$ belong now to $\{t,z\}$). 

The two graded vector spaces $V$ and $W$ identify, but we distinguish them notationally as the two cofree graded coalgebras $T(V)$ and $T(W)$ do not. This is the case in particular  because the two projections $\pi_V$ and $\pi_W$ are different: for example, $\pi_W(x^\ast)=x^\ast$, $\pi_W(y^\ast)=z^\ast$, $\pi_W(x^\ast x^\ast)=z^\ast$ whereas $\pi_V(x^\ast)=x^\ast$, $\pi_V(y^\ast)=y^\ast$,$\pi_V(x^\ast x^\ast)=0$.
Conversely, $\pi_V(t^\ast)=x^\ast, \ \pi_V(z^\ast)=y^\ast,\ \pi_V(t^\ast t^\ast)=-y^\ast$. Both $\pi_V$ and $\pi_W$ are the null map on all words of degree greater or equal $3$.

We also get
\[t^\ast\times t^\ast=2t^\ast t^\ast+2z^\ast\]
as 
\[<t^\ast\times t^\ast|tt>=<t^\ast\otimes t^\ast|\Delta(tt)>=<t^\ast\otimes t^\ast|tt\otimes 1+2t\otimes t+1\otimes tt>=2\]
and
\[<t^\ast\times t^\ast|z>=<t^\ast\otimes t^\ast|\Delta(z)>=<t^\ast\otimes t^\ast|z\otimes 1+2t\otimes t+1\otimes z>=2.\]
One obtains thus using $\pi_W$ another graded commutative $B_\infty$-algebra structure  $\langle-,-\rangle'$ on $W$ which however is {\it not} as a Hopf algebra the shuffle algebra over $t^\ast$ and $z^\ast$. 
We indeed obtain that $\langle t^\ast,t^\ast\rangle'=2z^\ast$, and $\langle- ,-\rangle'$ is $0$ on all other pairs of words in $t^\ast$ and $z^\ast$. 
This illustrates the general idea that a cofree graded commutative Hopf algebra structure on $H$ is actually obtained as a graded connected commutative Hopf algebra structure {\it plus} the choice of a basis of words making it a cofree coalgebra over the corresponding letters (up to the choice of a basis of $\prim(H)$). There are infinitely many such choices (the basis of $\prim(H)$ being fixed), as our example implies. 

The change of basis in $A$ from words in $x,y$ to words in $t,z$ is obtained simply by substituting $t$ for $x$ and $z-tt$ for $y$ (and conversely $x$ for $t$ and $y+xx$ for $z$). The change of basis in $A^\ast$ from words in $x^\ast,y^\ast$ to words in $t^\ast,z^\ast$ is slightly more delicate but follows directly from Lemma \ref{cofreeuniv}: it is given by
$$y_1^\ast\ldots y_n^\ast\longmapsto \sum\limits_{k\in\mathbb N^*}\pi_W^{\otimes k}\circ\overline\Delta_k(y_1^\ast\ldots y_n^\ast)$$
(with the $y_i$ in $\{x,y\}$) and the inverse map by 
$$y_1^\ast\ldots y_n^\ast\longmapsto \sum\limits_{k\in\mathbb N^*}\pi_V^{\otimes k}\circ\overline\Delta_k(y_1^\ast\ldots y_n^\ast)$$
(with the $y_i$ in $\{t,z\}$).
For example, using this rule,
\begin{align*}
x^\ast y^\ast x^\ast x^\ast y^\ast &= \pi_W^{\otimes 4}(x^\ast\otimes y^\ast\otimes x^\ast x^\ast\otimes y^\ast)+\pi_W^{\otimes 5}(x^\ast\otimes y^\ast\otimes x^\ast\otimes x^\ast\otimes y^\ast)\\
&=t^\ast z^\ast z^\ast z^\ast+t^\ast z^\ast t^\ast t^\ast z^\ast .
\end{align*}

\end{example}
\begin{example}[Shuffle and quasi-shuffle Hopf algebra]\label{shuffle}
Let $\cdot$ be an associative, not necessarily unitary, product on $V$. We extend it to a $B_\infty$-structure on $V$ by putting, for any words
$w,w'\in T(V)$,
\[\langle w,w'\rangle=\begin{cases}
0\mbox{ if }w=w'=1,\\
w'\mbox{ if }w=1\mbox{ and }\ell(w')=1,\\
w\mbox{ if }\ell(w)=1\mbox{ and }w'=1,\\
w\cdot w'\mbox{ if }\ell(w)=\ell(w)'=1,\\
0\mbox{ otherwise},
\end{cases}\]
where $\ell(w)$ stands for the length of the word $w$.
The associated product on $T(V)$ is called the quasi-shuffle product $\squplus$. It equips $T(V)$ with a Hopf algebra structure.
It is commutative, if, and only if, $\cdot$ is commutative.
In the particular case where $\cdot=0$, we obtain the shuffle product $\shuffle$ (see also Remark \ref{rkrew}).

The quasi-shuffle product is classically inductively (and equivalently) defined by the equations $v\squplus 1=1\squplus v=v$ and
\begin{align*}
v_1\ldots v_k\squplus v_{k+1}\ldots v_{k+l}&:=v_1(v_2\ldots v_k\squplus v_{k+1}\ldots v_{k+l})\\
&+v_{k+1}(v_1\ldots v_k\squplus v_{k+2}\ldots v_{k+l})\\
&+(v_1\cdot v_{k+1})(v_2\ldots v_k\squplus v_{k+2}\ldots v_{k+l}),
\end{align*}
which restricts to 
\begin{align*}
v_1\ldots v_k\shuffle v_{k+1}\ldots v_{k+l}&:=v_1(v_2\ldots v_k\shuffle v_{k+1}\ldots v_{k+l})\\
&+v_{k+1}(v_1\ldots v_k\shuffle v_{k+2}\ldots v_{k+l}),
\end{align*}
for the shuffle product.
See also e.g. \cite{foissy2} for details on shuffle and quasi-shuffle algebras and their relationships, to be generalized below in the present article.

Commutative quasi-shuffle algebras are at the moment the most important example of commutative $B_\infty$-algebras. They can be used to encode certain relations between multi-zeta values and generalizations of them, see \cite{Hoffman2,Hoffman3} for a review on this topic.
They play a prominent role in the study of Rota-Baxter algebras \cite{Guo,Guo2,EFG}.
They are used not only in combinatorics and algebra as they have for example applications also in stochastics, where they allow to better understand the equivalence between It\^o and Stratonovich integrals, resp. solutions of stochastic differential equations (It\^o calculus being encoded by quasi-shuffle algebras and Stratonovich's by shuffle algebras) \cite{ebrahimi2015flows,ebrahimi2015exponential}.
\end{example}

\begin{example}[Graded shuffle and quasi-shuffle Hopf algebra]\label{gshuffle}
Assume now that $V=\bigoplus\limits_{n\geq 1}V_n$ is a graded vector space. Let then  $\cdot$ be an associative, commutative and graded product on $V$.
The standard example is the case where $V$ is the semigroup ring of a positively graded commutative semigroup.

The degree of an element $v$ in $V_n$ is denoted $|v|$ ($|v|:=n$). Given $v_i\in V_{k_i},\ i=1\dots n$, the $qs$-degree (quasi-shuffle degree) of $v_1\ldots v_n\in V^{\otimes n}$ is, by definition, $|v_1\ldots v_n|:=k_1+\ldots +k_n$. This is the grading we have previously introduced on $T^g(V)$.
Direct inspection shows that the $qs$-degree defines a cofree graded commutative Hopf algebra structure on the $(T(V),\squplus,\Delta)$, where $\squplus$ was defined in Example \ref{shuffle}, making $V$ a graded commutative $B_\infty$-algebra. 

One can show that there is an isomorphism (of Hopf algebras) between the shuffle algebra and the quasi-shuffle algebra of words over an arbitrary graded commutative semigroup. This isomorphism is known as Hoffman's isomorphism. It generalizes to the non graded case.
See e.g. \cite{Hoffman} for details on graded quasi-shuffle algebras and their relationship with shuffle algebras, to be generalized below in the present article. Graded quasi-shuffle and shuffle algebras are used in \cite{BFT} to investigate It\^o to Stratonovich transformations in the context of branched rough paths.
\end{example}

\begin{example}[The gebra of descents]\label{dsg}
Let us first introduce briefly the context of the following results. Together with Example \ref{flalg} they will serve as a benchmark for our later developments on graded commutative $B_\infty$-algebras.

A classical and fundamental theorem by C. Malvenuto in the theory of free Lie algebras, symmetric group representations and symmetric functions asserts that the dual of the Hopf algebra of descents in symmetric groups, whose definition is recalled below, is isomorphic to the Hopf algebra of quasi-symmetric functions \cite{Mal,MR}. There are several ways to understand this isomorphism, but one of them is particularly important for our purposes as it is the pattern we will use and generalize to investigate the structure of graded commutative $B_\infty$-algebras.

We postpone details to latter developments in the article, but sketch the fundamental idea. The Hopf algebra of descents $\desc$, whose definition is recalled below, is naturally graded, cocommutative.
As an algebra it is a free associative algebra and the exponential allows to move between a family of primitive generators to a family of group-like generators.  When dualizing, the  choice of a family of primitive generators amounts to considering the dual as equipped with a shuffle Hopf algebra structure, in the sense that the formulas for the product {\it in the corresponding basis} are the usual formulas for shuffle products. Choosing group-like generators amounts instead amounts to equip the dual with a specific graded commutative $B_\infty$-structure: up to isomorphism the one of the quasi-shuffle algebra over the semigroup algebra over the positive integers.
A change of basis in $\desc$ and its graded dual appears therefore as a prototype example for Hoffman's isomorphisms between shuffle and quasi-shuffle Hopf algebras, but there is more to be learned from that example, as we show now.

Let us explain the technical content of these ideas more precisely and recall the definition of the Hopf algebra of descents and some of its key properties, relevant to the present article. The reader is referred to \cite[Chap. 5]{CP21}, from which the following definitions and properties are taken, for more details.

\begin{defi}[Descent sets of permutations]\index{Descent! set} 
A permutation $\sigma$ in the $n$-th symmetric group $S_n$ is said to have a descent in position $i<n$ if and only if $\sigma(i)>\sigma(i+1)$. The set of descents of a permutation is denoted $desc(\sigma)$:
$$desc(\sigma):=\{i<n, \sigma(i)>\sigma(i+1)\}.$$
\end{defi}

To each  subset $S=S'\sqcup\{n\}$ of $[n]$ containing $n$ are associated two elements in the group algebra $\mathbb Q [S_n]$:
$$
	De_{=S}:=\sum\limits_{\sigma\in S_n \atop desc(\sigma)=S'}\sigma \ ,\quad 
	De_{S}:=\sum\limits_{\sigma\in S_n \atop desc(\sigma)\subset S'}\sigma.
$$
The $De_S$ (resp. the $De_{=S}$), $S=S'\sqcup\{n\}\subset [n]$ are linearly independent in $\mathbb Q [S_n]$. They have the same linear span, written $\desc$.

Let now $X=\{x_1,\dots,x_n,\dots\}$ be a countable alphabet and $T(X)$ the tensor gebra over $X$ equipped with a graded connected cocommutative Hopf algebra structure by the concatenation product and the unshuffle coproduct $\Delta_\shuffle$ (the coproduct dual to the shuffle product of words: $\Delta_\shuffle(y_1\dots y_n):=\sum\limits_{I\coprod J=[n]}y_I\otimes y_J$ where, if $I=\{i_1,\dots,i_k\}, y_I:=y_{i_1}\dots y_{i_k}$). 
Permutations in $S_n$ act on $T_n(X)$ on the right by permutation of the letters of words of length $n$ ($\sigma(y_1\dots y_n)=y_{\sigma(1)}\dots y_{\sigma(n)}$ where $y_i\in X,\ i\leq n$). The convolution product of linear endomorphisms of $T(X)$ induces then a graded algebra structure on the direct sum of linear spans of the symmetric groups $\bigoplus\limits_{n\in \mathbb N}\mathbb Q[S_n]$ that restricts to a graded algebra structure on $\desc$ (elements in $S_n$ being of degree $n$).
This is the standard way to connect the combinatorics of descents with the theory of free Lie algebras \cite{Reutenauer}. 

From the action of $\desc$ on $T(X)$ (by a process that holds actually more generally for all graded connected cocommutative Hopf algebras), one can derive the existence of a graded cocommutative Hopf algebra structure on $\desc$ and one gets the Theorem:

\begin{theo}
The descent algebra $\desc$ is a graded cocommutative Hopf algebra, freely generated as a unital associative algebra by any of the following families:
\begin{itemize}
\item The identity permutations $1_n$ in the groups $S_n$, $n\geq 1$, that form a group-like family (that is to say, for any $n\in \N$, $\Delta(1_n)=\sum\limits_{k=0}^n1_k\otimes 1_{n-k}$),
\item The Dynkin operators 
$$Dyn_n=\sum_{i=0}^{n-1}(-1)^iDe_{=\{1,\dots,i\}},$$
which are primitive elements,
\item Solomon's Eulerian idempotents
$$sol_n:=e_n^1=\sum\limits_{S=S'\sqcup\{n\}\subset[n]}\frac{(-1)^{|S'|}}{n}{n-1\choose |S'|}^{-1}De_{S},$$
which are primitive elements.
\end{itemize}
\end{theo}

The Theorem implies that the family of Dynkin operators and the family of Eulerian idempotents define a standard free-Lie type Hopf algebra structure on $\desc$ (they are primitive elements and freely generate $\desc$ as an associative algebra). The graded dual Hopf algebra $\desc^*$ is a shuffle Hopf algebra in the corresponding bases. 

This is a general phenomenon: any family of Lie idempotents (elements in the descent algebra that project the tensor algebra onto the free Lie algebra) would define a shuffle Hopf algebra structure on $\desc$, and there are infinitely many of them. On Lie idempotents and their generalization to arbitrary connected cocommutative Hopf algebras, see \cite{PR99,PR02,PR02b}.

On the other hand, the family of the identity elements in symmetric groups defines a non standard free-Lie type Hopf algebra structure on $\desc$. Dualizing, one gets a graded commutative $B_\infty$-algebra structure on the space $\prim(\desc^*)$ whose associated Hopf algebra is not a shuffle Hopf algebra over the cogenerating vector space $\hat W^\ast$, where $W$ is the linear span of the identity elements in symmetric groups. The graded dual of $\desc$ is indeed a quasi-shuffle Hopf algebra over $\hat W^\ast$ --- see \cite{Mal,MR}. 

Choosing a family of generators of $\desc$ that would not be primitive nor group-like would lead to a graded commutative $B_\infty$-algebra structure on $\prim(\desc^*)$ whose associated Hopf algebra would not be a shuffle Hopf algebra nor a quasi-shuffle Hopf algebra in the corresponding basis. This also follows easily from the same references and can be checked directly --- we don't detail this point as the reason for such a claim should be clear from forthcoming developments.

\end{example}

\begin{example}[Natural deformations of shuffle Hopf algebras]
In a joint article with J.-Y. Thibon, we investigated and classified more generally natural deformations of the shuffle Hopf algebra
structure $\sh(A)$ which can be defined on the space of tensors over a
commutative algebra A (where natural means functorial) \cite{foissy2}. These deformations are parametrized by formal power series. To each such non trivial deformation corresponds a {\it non graded} commutative $B_\infty$-algebra structure on $A$. We refer the reader to our article for definitions and details (we do not use the language of $B_\infty$-algebras in that article but it should be clear how to interpret our results in these terms).
\end{example}

\begin{example}[Finite topologies]
In another direction, together with C. Malvenuto, we investigated commutative $B_\infty$ structures in the context of finite topologies (or, equivalently, preorders) and their links with shuffle Hopf algebras, introducing and featuring in particular the notion of Schur-Weyl categories of bialgebras \cite{FMP}. The fact that such structures appear very naturally (but in a non straightforward way) in the context of finite topologies and ordered structures provides, together with the homotopical algebra origin of these notions in \cite{GJ94}, further evidence for their naturalness.
More details will be given in section \ref{sectiontopo}.
\end{example}

\section{The structure of graded commutative $B_\infty$-algebras}\label{SGB}

Recall first the Poincar\'e-Birkhoff-Witt (PBW) theorem (we refer again to \cite{CP21} for details on the materials that follow). Let $L$ be a graded and reduced Lie algebra (reduced meaning that it has no component in degree 0). Its enveloping algebra $H$ is a graded connected cocommutative Hopf algebra and there is a canonical morphism $\iota$ from the space $S(L)$ of symmetric tensors over $L$ to $H$. This map is a coalgebra isomorphism and the decomposition of $S(L)$ according to tensor degrees induces a decomposition of the graded components of the enveloping algebra: $H_n=\bigoplus\limits_{k\leq n}H_n\cap S^k(L)$,
where we write $S^k(L)$ for the space of symmetric tensors in $L^{\otimes k}$.

The gebra of tensors $T(X)$ over a set $X$ is, when equipped with the concatenation product and the unshuffle coproduct $\Delta_{\shuffle}$, the enveloping algebra of the free Lie algebra over $X$. 
One of the key properties of
the Eulerian idempotents is that they project $T(X)$ to the free Lie algebra over $X$ according to the decomposition of $T(X)$ induced by the Poincar\'e-Birkhoff-Witt theorem (this is actually how Solomon defined them originally). This idea was generalized to arbitrary graded connected cocommutative or commutative Hopf algebras $H$ in \cite[Th. I,5,6 and Th. I,6,4]{patras1992}, see also \cite{patras1993decomposition,patras1994algebre,CP21}. 

In the cocommutative case, this leads to the definition of generalized Eulerian idempotents that project on the graded components of primitive part of the Hopf algebra in agreement with the Poincar\'e-Birkhoff-Witt decomposition. We will still write $e_n^1$ for the generalized Eulerian idempotent, acting on $H_n$ as a projector from $H_n$ to $H_n\cap \prim(H)$. Notice that these generalized idempotents {\it cannot} be constructed in general as elements of the symmetric group algebras. In particular they are not directly governed by the statistics of descents in symmetric groups --- the very reason for the name ``Eulerian idempotents'' used in the classical case. This is the reason why they should be preferably called ``canonical idempotents'' --- we use below the two terminologies indifferently.

\begin{lemma}\label{pbw}
Given $x\in H_n$, where $H$ is a graded connected cocommutative Hopf algebra, the PBW theorem induces a unique decomposition
$x=e_n^1(x)+y$, where $e_n^1(x)$ is a primitive element in $H_n$ and $y\in\bigoplus\limits_{k>1}H_n\cap S^k(L)$. In particular, $y$ can be expanded as a sum of products of elements in $\bigoplus\limits_{k<n}H_k$.
\end{lemma}

More generally, one can show \cite[Th. 5.2.1]{CP21} that any primitive element $\psi_n$ in the $n$-th graded component of the descent algebra ($\psi_n\in Prim_n(\desc):=\prim(\desc)\cap \desc_n$) defines a projector from $H_n$ onto $\prim(H)\cap H_n$ provided the coefficient of $1_n\in S_n$ in the expansion of $\psi_n$, viewed as an element of the group algebra of $S_n$, is 1 in the basis of permutations. 
We call a family $(\psi_n)_{n\in \N^\ast}$ of such projectors a Lie idempotent family. The Dynkin idempotents $Dyn_n/n$ and the Klyachko idempotents provide, together with the Eulerian idempotents, classical examples of such families. There are infinitely many as any convex combination of Lie idempotent families is a Lie idempotent family --- this follows from their characterization in terms of the coefficients of $1_n\in S_n$ in their expansion.

It is easy to show that Lemma \ref{pbw} generalizes to these families $(\psi_n)_{n\in \N^\ast}$ in the following way: given $x\in H_n,\ x=\psi_n(x)+y$, where $\psi_n(x)$ is a primitive element in $H_n$ and $y$ can be expanded as a sum of products of elements in $\bigoplus\limits_{k<n}H_k$.

\begin{theo}[Structure theorem for free-Lie type Hopf algebras]\label{stf}
Let $H$ be a free-Lie type Hopf algebra with freely generating subspace $W$. Then, for any Lie idempotent family $(\psi_n)_{n\in \N^\ast}$
$H$ is a standard free-Lie type Hopf algebra over the freely generating subspace $W':=\bigoplus\limits_{n\in\N^\ast}\psi_n(W_n)$. In particular, a free-Lie type Hopf algebra is always naturally the enveloping algebra of a free Lie algebra over a generating subspace $W'$, and any Lie idempotent family gives rise to such a space.
\end{theo}
\begin{proof}
Indeed, $W$ freely generates $H$ as a free associative algebra. Let us choose a graded basis $(b_1,\ldots,b_n,\ldots )$, ordered in such a way that $|b_{i+1}|\geq |b_i|$.
By Lemma \ref{pbw}, $b_i=\psi_n(b_i)+r$, where the residuum $r$ belongs to the free associative algebra generated by the $b_j$, $j<i-1$. By the standard triangularity argument, $\psi_n(b)$ is freely independent from the $b_j$, $j<i$ and from the $\psi_n(b_j),\ j\leq i$, and the $\psi_n(b_i)$ freely generate $H$. The Lemma follows.
\end{proof}
The theorem can be recast in categorical terms. Let $V$ a graded vector space together with an isomorphism $\phi:V\to W$. We write $\phi_n$ for the degree $n$ component of the isomorphism. The data $(H,W)$ of a free-Lie type Hopf algebra are equivalent to the data $(H,V,\phi)$ of a graded cocommutative Hopf algebra structure on $H$ together with the linear injection (still written) $\phi$ from $V$ into $H$ that induces an algebra isomorphism $T(V)\cong H$. We call $(H,V,\phi)$ a presentation of the  free-Lie type Hopf algebra $(H,W)$.

\begin{theo}[Structure theorem \ref{stf}, categorical formulation]
Let $(H,V,\phi)$ be a presentation of a free-Lie type Hopf algebra $H$. Then,  for any Lie idempotent family $(\psi_n)_{n\in \N^\ast}$, $(H,V,\bigoplus\limits_{n\in\N^\ast}\psi_n\circ \phi_n)$ is a presentation of a standard free-Lie type Hopf algebra structure on $H$. 
\end{theo}

The first formulation of the Theorem is a change-of-basis argument generalizing the one we encountered it in the Example of the Hopf algebra of descents. The second formulation is more abstract and says that a free-Lie type Hopf algebra is always canonically isomorphic to a standard free-Lie type Hopf algebra. This is the argument we encountered in the Example of graded quasi-shuffle Hopf algebras with Hoffman's isomorphism. 

By duality (hereafter in this section all duals are graded duals and all graded vector spaces are locally finite), we immediately get:
\begin{theo}[Structure theorem for cofree graded commutative Hopf algebras]
Let $H^\ast$ be a cofree graded commutative Hopf algebra with structure map $\phi^\ast:H^\ast\to \prim(H^\ast)$. Then,  for any Lie idempotent family $(\psi_n)_{n\in \N^\ast}$, the structure map $\pi:=\sum\limits_{n\in\N^\ast}\phi_n^\ast\circ \psi_n^\ast$ equips $H^\ast$ with another, isomorphic, cofree coalgebra structure over $\prim(H^\ast)$. Furthermore, the structure map $\pi$ induces a Hopf algebra isomorphism with a shuffle Hopf algebra $H^\ast\cong (T^g(\prim(H^\ast)),\shuffle,\Delta)$.
\end{theo}

The last sentence follows from Theorem \ref{stf} and the fact that the graded dual Hopf algebra of $(T^g(\prim(H^\ast)),\shuffle,\Delta)$ is the enveloping algebra of the free Lie algebra over the dual of $\prim(H^\ast)$.

The theorem can be expressed in the language of $B_\infty$-algebras. 

\begin{theo}[Structure theorem for graded commutative $B_\infty$-algebras]
Let $V^\ast$ be a graded commutative $B_\infty$-algebra and $(T^g(V^\ast),\ast,\Delta)$ the associated graded commutative Hopf algebra. Let $(\psi_n)_{n\in \N^\ast}$ be a Lie idempotent family.
Then, the structure map $\sum\limits_{n\in\N^\ast}\pi_V\circ \psi_n^\ast$ from $T^g(V^\ast)$ to $V^\ast$ induces an Hopf algebra isomorphism  $(T^g(V^\ast),\ast,\Delta)\cong  (T^g(V^\ast),\shuffle,\Delta)$.
\end{theo}

When the Lie idempotent family is the Eulerian family, the Theorem was obtained by Bellingeri, Ferrucci and Tapia in \cite[Rmk 3.5]{BFT}.

\section{The structure of commutative $B_\infty$-algebras}\label{SCB}

Let us start this Section by restating the characterization of shuffle Hopf algebras (Lemma \ref{trivial}).
\begin{lemma}\label{carshu}
Let $(H,*,\Delta,\pi)$ be a cofree commutative Hopf algebra. Then, $H$ is a shuffle algebra in the $\pi$ basis if and only if $\pi$ vanishes on $\bar H*\bar H$, the square of the augmentation ideal of $\bar H$.
\end{lemma}
In particular, any surjection $\gamma: H\to \prim(H)$ that acts as the identity map on $\prim(H)$ and vanishes on $\bar H*\bar H$ defines a shuffle algebra structure on $H$ in the $\gamma$ basis.

Recall (details can be found in \cite[Sections 2.10 and 3.3]{CP21}) that to a commutative Hopf algebra $H$ are classically associated a group and a Lie algebra: the group is the set of algebra maps from $H$ to the ground field $\mathbb K$ equipped with the restriction of the convolution product on $End(H)$, the algebra of linear endomorphisms of $H$. The Lie algebra is the vector space of linear forms on $\bar H$ that vanish on $\bar H*\bar H$ or, equivalently, of linear forms on $H$ that vanish on
$\K+\bar H*\bar H$. These linear forms are usually called infinitesimal characters, their bracket is obtained as the bracket associated to the convolution product. 
This construction generalizes from linear forms to linear endomorphisms of $H$. A linear endomorphism of $H$ that vanishes on $\mathbb K+\bar H*\bar H$ is called an infinitesimal endomorphism of $H$. 

\begin{defi}
An infinitesimal endomorphism $\phi$ of $H$ is called tangent to identity if and only if its restriction to $\prim(H)$ is the identity map. 
\end{defi}

We state the following Corollary of Lemma \ref{carshu} as a Theorem in view of its meaningfulness for the theory.

\begin{theo}\label{varpibasis}
Let $(H,*,\Delta,\pi)$ be a cofree commutative Hopf algebra and $\phi$ a tangent to identity endomorphism of $H$.
By Lemma \ref{carshu}, $H$ is a shuffle Hopf algebra in the $\pi\circ \phi$ basis.
\end{theo}

When restated in the language of $B_\infty$-algebras, the Theorem reads:
\begin{theo}\label{isoH2}
Let $\langle-,-\rangle$ be a commutative $B_\infty$ structure on $V$ with associated Hopf algebra $(T(V),*,\Delta)$. Let $\phi$ a tangent to identity endomorphism of $T(V)$ and 
let $\tilde\omega$ be the coalgebra automorphism of $(T(V),\Delta)$ induced by $\omega:=\pi_V\circ\phi$: 
\[\tilde\omega:\left\{\begin{array}{rcl}
T(V)&\longrightarrow&T(V)\\
1&\longmapsto&1,\\
w\in T_+(V)&\longmapsto&\displaystyle \sum_{k=1}^\infty \sum_{\substack{w=w_1\ldots w_k,\\ w_1,\ldots,w_k\neq \emptyset}}
\omega(w_1)\ldots \omega(w_k). 
\end{array}\right.\]
Then $\tilde \omega$ is a Hopf algebra isomorphism from $(T(V),*,\Delta)$ to $(T(V),\shuffle,\Delta)$. 
\end{theo}

The calculation of the inverse isomorphism follows from the computation of the inverse of a cofree coalgebra isomorphism in Eq. \ref{inversestru}.

\begin{prop}\label{isoL} Let notation be as in Theorem \ref{isoH2}.
Let us define inductively $\zeta$ (by induction on the length of tensors) by
\[\zeta:\left\{\begin{array}{rcl}
T_+(V)&\longrightarrow&V\\
v\in V&\longmapsto&v,\\
w\in T_n(V),\ n\geq 2&\longmapsto& \zeta(w)= -\displaystyle \sum_{k=2}^{n} \sum_{\substack{w=w_1\ldots w_k,\\ w_1,\ldots,w_k\neq \emptyset}}
\varpi_*(\zeta(w_1)\ldots \zeta(w_k)). 
\end{array}\right.\]

Let us then define 
\[\tilde\zeta:\left\{\begin{array}{rcl}
T(V)&\longrightarrow&T(V)\\
1&\longmapsto&1,\\
w\in T_+(V)&\longmapsto&\displaystyle \sum_{k=1}^\infty \sum_{\substack{w=w_1\ldots w_k,\\ w_1,\ldots,w_k\neq \emptyset}}
\zeta(w_1)\ldots \zeta(w_k). 
\end{array}\right.\]
Then $\tilde\zeta$ is the Hopf algebra isomorphism from to $(T(V),\shuffle,\Delta)$ to $(T(V),*,\Delta)$ inverse to $\tilde\omega$ as defined in Theorem \ref{isoH2}. 
\end{prop}

In the graded case we saw that any Lie idempotent family defines a shuffle Hopf algebra structure on a cofree graded commutative Hopf algebra.
In the non graded case, the key idea to construct a universal tangent to identity endomorphism will be to use the extension of the definition of the Eulerian idempotents from the original case of the tensor Hopf algebra to the case where the Hopf algebra $H$ is commutative and unipotent (being unipotent is a weaker hypothesis than being graded connected; it always holds when the Hopf algebra is conilpotent as a coalgebra).  
We have already used in previous works on quasi-shuffle algebras this fact that the constructions and proofs of structure results on graded connected cocommutative or commutative Hopf algebras in \cite{patras1992,patras1993decomposition,patras1994algebre} can be extended to a broader setting as they actually only require the Hopf algebras to be unipotent  --- this observation was developed systematically in \cite[Chap. 4]{CP21}, to which we refer for details and proofs. See in particular \cite[Thm 4.4.1]{CP21}.

When applied to commutative $B_\infty$-algebras, these results imply the Theorem:
\begin{theo}\label{canid}
Let $*\in P^c(V)$. The canonical idempotent $e_*$ is defined by $e_*(1)=0$ and for any non-empty word $w$ of length $n$,
\[e_*(w) 
=\sum_{k=1}^n \sum_{\substack{w=w_1\ldots w_k,\\ w_1,\ldots,w_k\neq \emptyset}}
\dfrac{(-1)^{k-1}}{k} w_1*\ldots*w_k.\]
Then $e_*$ is a projector, vanishing on $T_+(V)*T_+(V)$, and for any $v\in V$, $e_*(v)=v$: it is a tangent to identity infinitesimal endomorphism of $T(V)$. Moreover, the image of $e_*$ freely generates $(T(V),*)$ as a commutative algebra.

Furthermore, as the product of a $B_\infty$-algebra is a morphism of filtered graded coalgebras,
$$e_*(T_n(V))\in \bigoplus\limits_{k=1}^nT_k(V).$$
\end{theo}

When the same results are applied more generally to a cofree commutative Hopf algebra $H$, the same conclusions hold {\it mutatis mutandis}. The canonical (or generalized Eulerian) idempotent is then defined on $\bar H$ by
\[e_*(h) 
=\sum_{k=1}^\infty
\dfrac{(-1)^{k-1}}{k} m_k\circ\bar\Delta_k(h),\]
where we write $m_k$ for the iterated product (from $\bar H^{\otimes k}$ to $H$). 
It
projects onto a (canonically constructed) vector subspace $Q(H)$ of $H$ that freely generates $H$ as a commutative algebra. The projection is orthogonal to the square of $\bar H$, the augmentation ideal of $H$ and acts as the identity on $\prim(H)$. It is a tangent to identity infinitesimal endomorphism of $H$.

\begin{prop}\label{varpidem}
Let $*\in P^c(V)$.
We set $\varpi_*:=\pi_V\circ e_*$ and call $\varpi$ the canonical commutative $B_\infty$ idempotent. It acts as the identity map on $V$ and sends any non-empty word $w$ to
\[\varpi_*(w)=\sum_{k=1}^\infty \sum_{\substack{w=w_1\ldots w_k,\\ w_1,\ldots,w_k\neq \emptyset}}
\dfrac{(-1)^{k-1}}{k} \langle w_1,w_2*\ldots*w_k\rangle.\]
\end{prop}
\begin{proof}
The map $\varpi_*$ is indeed an idempotent since $\pi_V$ is an idempotent and $e_\ast\circ\pi_V=\pi_V$. The explicit formula follows from the formula for $e_*$ in Theorem \ref{canid} and the observation that, since $\langle w,w'\rangle =\pi_V(w*w')$,
\[\pi_V(w_1*w_2*\ldots*w_k)=\pi_V(w_1*(w_2*\ldots*w_k))=\langle w_1,w_2*\ldots*w_k\rangle. \qedhere\]
\end{proof}

\begin{example}
Let $v_1,v_2,v_3\in V$.
\begin{align*}
\varpi_*(v_1)&=v_1,\\
\varpi_*(v_1v_2)&=-\dfrac{1}{2}\langle v_1,v_2\rangle,\\
\varpi_*(v_1v_2v_3)&=-\dfrac{1}{2}\left(\langle v_1v_2,v_3\rangle+\langle v_1,v_2v_3\rangle\right)
+\dfrac{1}{3}\langle v_1,v_2v_3+v_3v_2+\langle v_2,v_3\rangle\rangle.
\end{align*}
\end{example}

\begin{prop} When $(H,\pi)$ is a cofree commutative Hopf algebra, define similarly the idempotent $\varpi_*:=\pi\circ e_*$ and call it the canonical $\pi$-idempotent. The Hopf algebra $H$ is then a shuffle Hopf algebra in the $\varpi$ basis.
\end{prop}
\begin{proof}
The Proposition follows from the fact that $e_*$ is a tangent to identity endomorphism and from Theorem \ref{varpibasis}.
\end{proof}
In the language of $B_\infty$-algebras, the Proposition reads:
\begin{prop}\label{isoshuf} Let notation be as in Proposition \ref{isoH2} but set $\varpi:=\pi_V\circ e_*$.
Then, $\tilde \varpi$ is a Hopf algebra isomorphism from $(T(V),*,\Delta)$ to $(T(V),\shuffle,\Delta)$. 
\end{prop}
 
\begin{remark}\label{qshvarpi}
In the quasi-shuffle case, the formulas simplify.  
Given two non empty words $w$ and $w'$, $\langle w,w'\rangle$ vanishes excepted when $w$ and $w'$ are both of length 1. Therefore, using also that the quasi-shuffle product of two words of length $p$ and $q$ is a linear combination of words of length at least $\max(p,q)$, we get:
\[\varpi_*(v_1\ldots v_n)=
\dfrac{(-1)^{n-1}}{n} \langle v_1,v_2\squplus\ldots\squplus v_n\rangle.\]
\[=\dfrac{(-1)^{n-1}}{n}\langle v_1,v_2\cdot\ldots\cdot v_n\rangle\]
\[= \dfrac{(-1)^{n-1}}{n}v_1\cdot v_2\cdot\ldots\cdot v_n,\]
and we obtain that for any $v_1,\ldots,v_n \in V$, with $n\geq 1$,
\[\varpi_*(v_1\ldots v_n)=\dfrac{(-1)^{n-1}}{n} v_1\cdot \ldots \cdot v_n.\]
Applying Proposition \ref{isoshuf}, we recover the Hoffman ``logarithmic'' isomorphism from $(T(V),\squplus,\Delta)$ to $(T(V),\shuffle,\Delta)$ \cite{Hoffman}.

The formulas for the inverse map (Proposition \ref{isoL}) also
simplify and a closed formula for $\zeta$ and $\tilde\zeta :=\tilde\varpi^{-1}$ can be obtained. 
Indeed, using Remark \ref{qshvarpi}, we get in that case the inductive formula 
\[\zeta(w)= -\displaystyle \sum_{k=2}^{n} \sum_{\substack{w=w_1\ldots w_k,\\ w_1,\ldots,w_k\neq \emptyset}}
\frac{(-1)^{k-1}}{k}\zeta(w_1)\cdot \ldots \cdot \zeta(w_k),\]
which, using for example the identity of coefficients resulting from the formal power series expansion of the identity $\mathrm{log}\circ \mathrm{exp}(x)=x$, is solved by
\[\zeta(w)=\frac{1}{n!}v_1\cdot \ldots \cdot v_n,\]
where $w=v_1\dots v_n\in T_n(V)$.
Applying Proposition \ref{isoL}, we recover Hoffman's ``exponential'' isomorphism from $(T(V),\shuffle,\Delta)$ to $(T(V),\squplus,\Delta)$ \cite{Hoffman}.

A similar analysis of the Hoffman isomorphism was performed in \cite{BFT}, where however the commutative algebra $V$ underlying the construction of the quasi-shuffle Hopf algebra is the algebra of  symmetric tensors over a vector space so as to obtain a graded commutative $B_\infty$ structure, as the article is written in the context of graded $B_\infty$-algebras. Using the unipotent version of structure theorems for Hopf algebras allows us to remove this restriction. 

See also our \cite{foissy2,FP20} where more advanced insights on the Hoffman isomorphism and more generally on deformations of shuffle Hopf algebras can be found.  
\end{remark}

We conclude with a remark on the naturality of the constructions presented in the article. Morphisms of commutative $B_\infty$-algebras are defined in the obvious way: given $(V,\langle-,-\rangle_V)$ and $(W, \langle-,-\rangle_W)$ two commutative $B_\infty$-algebras, a linear map $f$ from $V$ to $W$ is a commutative $B_\infty$ morphism if and only if, for any $v_1,\ldots,v_{p+q}$ in $V$,
\[\langle F(v_1\cdots v_p),F(v_{p+1}\cdots v_{p+q})\rangle_W = f(\langle v_1\otimes\cdots\otimes v_p,v_{p+1}\otimes\cdots\otimes v_{p+q}\rangle_V),\]
where $F$ is defined by $F(v_1\cdots v_i):=f(v_1)\otimes\cdots\otimes f(v_i)$.

We put $*_V=\Theta^{-1}(\langle-,-\rangle_V)$
and $*_W=\Theta^{-1}(\langle-,-\rangle_W)$ and let the reader check that if $f$ is a commutative $B_\infty$ morphism from $V$ to $W$, $F$ is a bialgebra morphism from $(T(V),*_V,\Delta)$ to $(T(W),*_W,\Delta)$.

\begin{prop}
Let $(V,\langle-,-\rangle_V)$ and $(W, \langle-,-\rangle_W)$ be two commutative $B_\infty$-algebras.  Let $f:V\longrightarrow W$ be a morphism of $B_\infty$-algebras. The following diagram is commutative:
\[\xymatrix{(T(V),*_V,\Delta) \ar[r]^{F} \ar[d]_{H_{*_V}} &(T(W),*_W,\Delta)\ar[d]^{H_{*_W}} \\
(T(V),\shuffle,\Delta)\ar[r]_F&(T(W),\shuffle,\Delta)}\]
In other terms, the commutative $B_\infty$/shuffle isomorphism is functorial.
\end{prop}
\begin{proof}
 The Proposition directly follows from the definition of the maps $H_*$, $\varpi$ and $ e_*$, from the definition of a commutative $B_\infty$ morphism, and the fact that $F$ is a morphism of bialgebras (so that its action commutes in particular with taking products or computing coproducts).
\end{proof}

\section{A $B_\infty$ structure on finite topologies}

\label{sectiontopo}

Let us detail now a particularly meaningful example: finite topologies (see also \cite{FMP}).

Let $E$ be a finite set. A topology on $E$ is a set $T$ of subsets of $E$ such that:
\begin{itemize}
\item $\emptyset$ , $E\in T$.
\item If $A,B\in T$, then $A\cup B\in T$ and $A\cap B\in T$.
\end{itemize}
By Alexandroff's theorem \cite{Alexandroff1937}, given a finite set $E$, the set of topologies on $E$ is in one-to-one correspondence with quasi-orders on $E$, that is to say transitive and reflexive relations on $E$:
given such a relation $\leq$ on $E$, the topology $T_\leq$ associated to $\leq$ is the set of subsets $O\subseteq E$ such that
\begin{align*}
&\forall x,y\in E,&x\in O\mbox{ and }x\leq y\Longrightarrow y\in O.
\end{align*} 
Conversely, if $T$ is a topology on a finite set $E$, it gives rise to a quasi-order on $E$ defined by
\begin{align*}
&\forall x,y\in E,&x\leq_T y&\Longleftrightarrow \mbox{ any $O\in T$ containing $x$ also contains $y$}.
\end{align*}

If $\leq$ is a quasi-order on $E$, we define an equivalence relation on $E$ by
\[x\sim_\leq y\Longleftrightarrow x\leq y\mbox{ and }y\leq x.\]
Thus, $E/\sim_\leq$ inherits an order $\overline{\leq}$, defined by
\[\overline{x}\overline{\leq}\overline{y}\Longleftrightarrow x \leq y.\]
We will further on abbreviate $\sim_{\leq_T}$ to $\sim_T$.

In the sequel, we shall represent isoclasses of finite topologies by the Hasse graphs of $(E/\sim_T,\overline{\leq_T})$, with indices representing the cardinalities of the classes of $\sim_{\leq_T}$,
when these cardinalities are not equal to $1$. Here are finite topologies of cardinality $\leq 4$:
\begin{align*}
1;&&\tun;&&\tdeux,\tun\tun,\tdun{$2$};&&
\ttroisun,\ptroisun,\ttroisdeux,\tdeux\tun,\tddeux{}{$2$},\tddeux{$2$}{},\tun\tdun{$2$},\tdun{$3$},
\end{align*}
\begin{align*}
&\tquatreun,\tquatredeux,\tquatrequatre,\tquatrecinq,\pquatreun,\pquatredeux,\pquatrequatre,\pquatrecinq,\pquatresept,\pquatrehuit,\ttroisun\tun,\ptroisun\tun,\ttroisdeux\tun,\tdeux\tdeux,\tdeux\tun\tun,\tun\tun\tun\tun,\\
&\tdtroisun{$2$}{}{},\tdtroisun{}{}{$2$},\pdtroisun{$2$}{},\pdtroisun{}{$2$}{},\tdtroisdeux{$2$}{}{},\tdtroisdeux{}{$2$}{},\tdtroisdeux{}{}{$2$},\tddeux{$2$}{}\tun,\tddeux{}{$2$}\tun,\tdeux\tdun{$2$},\tdun{$2$}\tun\tun,\tddeux{}{$3$},\tddeux{$3$}{},\tdun{$3$}\tun,\tddeux{$2$}{$2$},\tdun{$2$}\tdun{$2$},\tdun{$4$}.
\end{align*}
By convention, the edges in these graphs are oriented upwards.\\

The Hopf algebra $H_T$ of quasi-orders, or of finite topologies   \cite{FoissyMalvenuto,FMP} has for basis the set of (isoclasses) of finite topologies. Its product, that we will write $m$ further on, is given by the disjoint union: 
if $T$ and $T'$ are two topologies on respective sets $E$ and $E'$, then $TT'$ is a topology on the set $E\sqcup E'$, with 
\[TT'=\{O\sqcup O'\mid O\in T,\: O'\in T\}.\]
The Hasse graph of $TT'$ is the disjoint union of the Hasse graphs of $T$ and $T'$. This product is commutative, and its unit is the unique topology denoted $1$ on $\emptyset$. 
The coproduct $\Delta$ is defined on any finite topology $T$ on a set $E$ by
\[\Delta(T)=\sum_{O\in T} T_{\mid E\setminus O}\otimes T_{\mid O}.\]

\begin{example}\begin{align*}
\Delta(\tun)&=\tun\otimes 1+1\otimes \tun,\\
\Delta(\tdeux)&=\tdeux\otimes 1+1\otimes \tdeux+\tun\otimes \tun,\\
\Delta(\ttroisun)&=\ttroisun\otimes 1+1\otimes \ttroisun+\tdeux\otimes \tun+\tdeux\otimes \tun+\tun\otimes \tun\tun,\\
\Delta(\ptroisun)&=\ptroisun\otimes1+1\otimes \ptroisun+\tun\otimes \tdeux+\tun\otimes \tdeux+\tun\tun\otimes \tun,\\
\Delta(\ttroisdeux)&=\ttroisdeux\otimes 1+1\otimes \ttroisdeux+\tun\otimes \tdeux+\tdeux\otimes \tun.
\end{align*}\end{example}

The Hopf algebra $H_T$ is equipped with an extra structure: if $T$ and $T'$ are two topologies on respective finite sets $E$ and $E'$, then $T\downarrow T'$
is a topology on $E\sqcup E'$ defined by
\[T\downarrow T'=T'\cup \{O\sqcup E'\mid O\in T\}.\]
The Hasse graph of $T\downarrow T'$ is obtained by adding an edge from any maximal vertex of the Hasse graph of $T$ to any minimal vertex of the Hasse graph of $T'$. For example,
\begin{align*}
\tun \downarrow \tun&=\tdeux,&\tun \downarrow \tdeux&=\ttroisdeux,&\tun \downarrow \tun\tun&=\ttroisun,\\
&&\tdeux \downarrow\tun&=\ttroisdeux,&\tun\tun \downarrow\tun&=\ptroisun.
\end{align*}

The product $\downarrow$ is linearly extended to $H_T$, making it an associative algebra, which unit is again $1$. Moreover, by definition of $T\downarrow T'$, for any $x,y\in H_T$,
\[\Delta(x\downarrow y)=(x\otimes 1)\downarrow \Delta(y)+\Delta(x)\downarrow (1\otimes y)-x\otimes y.\]
Recall that an infinitesimal bialgebra in the sense of \cite{Loday2010} is
an associative unital algebra with product $\downarrow$ and a coassociative counital coalgebra with coproduct $\Delta$ and coaugmentation the unit of $\downarrow$ such that furthermore the previous identity is satisfied.
The triple $(H_T,\downarrow,\Delta)$ is thus an infinitesimal bialgebra. 

Let us survey some of the properties of these bialgebras with a view towards applications to $B_\infty$ structures. Complementary insights can be found in \cite{FMP}.
The main example of such objects are the tensor gebras $T(V)$, with the concatenation product that we write from now on $m_{conc}$ and the deconcatenation coproduct $\Delta$. 
In fact, in the connected case, these are the unique examples:

\begin{prop}\label{propinf}
Let $H=(H,\downarrow,\Delta)$ be an infinitesimal bialgebra. The following map is an injective map of infinitesimal bialgebra:
\[\Theta:\left\{\begin{array}{rcl}
(T(\prim(H)),m_{conc},\Delta)&\longrightarrow&(H,\downarrow,\Delta)\\
v_1\ldots v_k&\longmapsto&v_1\downarrow \ldots \downarrow v_k.
\end{array}
\right.\]
It is an isomorphism if and only if, the coalgebra $(H,\Delta)$ is conilpotent.
\end{prop}

\begin{proof}
The map $\Theta$ is obviously an algebra map. An easy induction on $k$ proves that for any $v_1,\ldots,v_k\in \prim(H)$, in $H$,
\[\Delta(v_1\downarrow \ldots \downarrow v_k)=\sum_{i=0}^k v_1\downarrow \ldots \downarrow v_i\otimes v_{i+1}\downarrow \ldots \downarrow v_k,\]
so $\Theta$ is a coalgebra morphism.

Let us assume that $\Theta$ is not injective. Let us consider $w\in \ker(\Theta)$, non zero. There exists $n\geq 1$ such that $w\in \displaystyle \bigoplus_{k=1}^n \prim(H)^{\otimes k}$.
Let us choose $w$ such that $n$ is minimal. Then the restriction of $\Theta$ to $\displaystyle \bigoplus_{k=1}^{n-1}$ is injective. Moreover, 
\begin{align*}
0&=\Delta \circ \Theta(w)\\
&=(\Theta\otimes \Theta)\circ \Delta(w)\\
&=\Theta(w)\otimes 1+1\otimes \Theta(w)+(\Theta\otimes \Theta)\overline{\Delta}(w)\\
&=(\Theta\otimes \Theta)\overline{\Delta}(w).
\end{align*}
Observing that 
\[\overline{\Delta}(w)\in \left(\bigoplus_{k=1}^{n-1} \prim(H)^{\otimes k}\right)^{\otimes 2},\]
we obtain that $\overline{\Delta}(w)=0$, so $\Delta(w)=w\otimes 1+1\otimes w$ and finally $w\in \prim(H)$. Therefore, $\Theta(w)=w=0$, which is a contradiction. So $\Theta$ is injective.\\

If $\Theta$ is surjective, then the coalgebras $(T(\prim(H)),\Delta)$ and $(H,\Delta)$ are isomorphic, so $H$ is conilpotent. 
Let us assume that $H$ is conilpotent and let us prove that $\Theta$ is surjective. Let $x\in Ker{\overline\Delta}_{n},\ n\geq 2$, let us prove that $x\in \mathrm{Im}(\Theta)$ by induction on $n$. 

If $n=2$, then $x\in \prim(H)$ and $x=\Theta(x)$. Let us assume the result at rank $n$ and let $x\in Ker{\overline\Delta}_{n+1}$.
Then, as  $\overline{\Delta}$ is coassociative, $\overline{\Delta}_n(x)\in \ker(\overline{\Delta})^{\otimes n}=\prim(H)^{\otimes n}$. Let us put 
\[\overline{\Delta}_n(x)=\sum_{i=1}^k v_{i,1}\otimes \ldots \otimes v_{i,n}.\]
Then, if $w=\displaystyle \sum_{i=1}^k  v_{i,1} \ldots v_{i,n}$, $\overline{\Delta}_n\circ \Theta(w)=\overline{\Delta}_n(x)$, so the induction hypothesis applies to $x-\Theta(w)$.
Therefore, $x-\Theta(w)\in \mathrm{Im}(\Theta)$ and finally $x\in \mathrm{Im}(\Theta)$. 
\end{proof}

\begin{prop}
Let $(H,\downarrow,\Delta)$ be a conilpotent infinitesimal bialgebra. Then 
\[\bar H=\mathrm{Prim}(H)\oplus \bar H\downarrow \bar H,\]
and the projection $\pi$ on $\mathrm{Prim}(H)$ in this direct sum is given by
\[\pi=\sum_{k=1}^\infty (-1)^{k+1}\downarrow^{(k-1)}\circ \overline{\Delta}^{(k-1)}.\]
\end{prop}

\begin{proof}
By Proposition \ref{propinf}, it is enough to prove it for $H=(T(V),m_{conc},\Delta)$. Let $n\geq 1$ and $w\in V^{\otimes n}$. 
\begin{align*}
\pi(w)&=\sum_{k\geq 1} \sum_{\substack{w=w_1\ldots w_k,\\ w_1,\ldots,w_k\neq 1}}(-1)^{k+1}w\\
&=\left(\sum_{k\geq 1} \sum_{\substack{w=w_1\ldots w_k,\\ w_1,\ldots,w_k\neq 1}}(-1)^{k+1}\right)w\\
&=\left(\sum_{I\subseteq [n-1]} (-1)^{|I|}\right)w\\
&=\delta_{n,1}w.
\end{align*}
Here, $I$ represents the places where the word $w$ is cut, where $i\in [n-1]$ reflects a cut between the letter $i$ and $i+1$ of $w$. 
Therefore, $\pi$ is a projection on $V=\mathrm{Prim}(T(V))$, which vanishes $T(V)_+\downarrow T(V)_+=\displaystyle \bigoplus_{n\geq 2} V^{\otimes n}$. 
\end{proof}

\begin{remark}
By Takeuchi's formula \cite{Takeuchi1971}, $\pi+\varepsilon$ is in fact the opposite of the antipode of $(H,\downarrow,\Delta)$, where by antipode is meant the convolution inverse of the identity map for the convolution product on linear endomorphisms induced by the coproduct $\Delta$ and the product$\downarrow$. This also follows more abstractly from the implicit definition of $\pi$ acting on $H=(T(V),m_{conc},\Delta)$ by the equation $\id-\varepsilon = \pi\ast \id$ or, equivalently, $(-\pi-\varepsilon )\ast \id =\varepsilon$.
\end{remark}

Therefore, if $(H,m,\Delta)$ is a commutative Hopf algebra, cofree as a coalgebra, with an extra product $\downarrow$ such that $(H,\downarrow,\Delta)$ is an infinitesimal bialgebra, 
then we can use this map $\pi$ as projector on $\prim(H)$.
This is the case for $H_T$. For example,
\begin{align*}
\pi(\tun)&=\tun,&
\pi(\tdeux)&=0,&
\pi(\tun\tun)&=\tun\tun-2\tdeux,\\
\pi(\ttroisun)&=0,&
\pi(\ptroisun)&=0,&
\pi(\ttroisdeux)&=0,\\
\pi(\tdeux\tun)&=\tdeux\tun-\ttroisun-\ptroisun+\ttroisdeux,\\
\pi(\tun\tun\tun)&=\tun\tun\tun-3\ttroisun-3\ptroisun+6\ttroisdeux.
\end{align*}
This projection induces a  commutative $B_\infty$ structure on $\prim(H_T)$, given by
\begin{align*}
\langle x_1\otimes \ldots \otimes x_k,y_1\otimes \ldots\otimes y_l\rangle&=\pi((x_1\downarrow \ldots \downarrow x_k)(y_1\downarrow \ldots \downarrow y_l)),
\end{align*}
where $x_1,\ldots,x_k,y_1,\ldots,y_l \in \prim(H_T)$. For example,
\begin{align*}
\langle \tun,\tun\rangle&=\pi(\tun\tun)=\tun\tun-2\tdeux,\\
\langle \tun\otimes \tun,\tun \rangle=\langle \tun,\tun\otimes \tun\rangle&=\pi(\tdeux\tun)=\tdeux\tun-\ttroisun-\ptroisun+\ttroisdeux,\\
\langle \tun\tun-2\tdeux,\tun \rangle=\langle \tun,\tun\tun-2\tdeux\rangle&=\pi(\tun\tun\tun-2\tdeux\tun)=\tun\tun\tun-2\tdeux\tun-\ttroisun-\ptroisun+4\ttroisdeux.
\end{align*}

\section{Finite topologies: the double bialgebra structure}
In this last section, we take advantage of the fact that the rich combinatorics of finite topologies allows to give an alternative calculation of the canonical (or generalized Eulerian) idempotents to illustrate what concrete computations can be performed to study changes of bases in commutative $B_\infty$-algebras.

The Hopf algebra $H_T$ also has a second coproduct $\delta$, defined with the help of the following notions \cite{Foissy37}. Let us consider a finite topology $T$, on a set $E$,
associated to the quasi-order $\leq_T$. Let $\sim$ be an equivalence relation on $E$.
\begin{enumerate}
\item $T\mid \sim$ is the topology associated to the quasi-order  $\leq_{T\mid \sim}$ defined on $E$ by
\[x \leq_{T\mid \sim} y\Longleftrightarrow x\leq_T y\mbox{ and }x\sim y.\]
\item $T/\sim$ is the topology associated to the quasi-order $\leq_{T/\sim}$ defined as the transitive closure of the relation defined  on $E$ by
\[x Ry\Longleftrightarrow x\leq_T y\mbox{ or }x\sim y.\]
\item We shall say that $\sim \in E_c(T)$ if:
\begin{itemize}
\item The connected components of $T\mid \sim$ are the equivalence classes of $\sim$.
\item The relation $\sim_{T/\sim}$ associated to the quasi-order $\leq_{T/\sim}$ is $\sim$.
\end{itemize}
\end{enumerate}
The coproduct $\delta$ sends any finite topology $T$ to
\[\delta(T)=\sum_{\sim\in E_c(T)} T/\sim \otimes T\mid \sim.\]
For example,
\begin{align*}
\delta(\tun)&=\tun\otimes \tun,\\
\delta(\tdeux)&=\tdeux\otimes \tun\tun+\tdun{$2$}\otimes \tdeux,\\
\delta(\ttroisun)&=\ttroisun \otimes \tun\tun\tun+2\tddeux{$2$}{} \otimes \tun\tdeux+\tdun{$3$}\otimes \ttroisun,\\
\delta(\ttroisdeux)&=\ttroisdeux \otimes \tun\tun\tun+\tddeux{$2$}{} \otimes \tun\tdeux++\tddeux{}{$2$} \otimes \tun\tdeux+\tdun{$3$}\otimes \ttroisdeux,\\
\delta(\ptroisun)&=\ptroisun \otimes \tun\tun\tun+2\tddeux{}{$2$} \otimes \tun\tdeux+\tdun{$3$}\otimes \ptroisun.
\end{align*}
The counit $\varepsilon_\delta$ of this coproduct sends any finite topology $T$ to $1$ if $T$ is discrete (that is to say if $\leq_T=\sim_T$) and to $0$ otherwise.
Then $(H_T,m,\Delta,\delta)$ is a double bialgebra, that is to say:
\begin{enumerate}
\item $(H_T,m,\delta)$ is a commutative bialgebra.
\item $(H_T,m,\Delta)$ is a commutative bialgebra in the category of right $(H_T,m,\delta)$-comodules with the coaction $\delta$, or in a more detailed version:
\begin{itemize}
\item $\Delta:H_T\longrightarrow H_T\otimes H_T$ is a comodule morphism, that is to say
\[(\Delta \otimes \id)\circ \delta=m_{1,3,24}\circ (\delta \otimes \delta)\circ \Delta,\]
where $m_{1,3,24}(h_1\otimes h_2\otimes h_3\otimes h_4):=h_1\otimes h_3\otimes m(h_3\otimes h_4).$
\item The counit $\varepsilon_\Delta:H_T\longrightarrow \K$ is a comodule morphism, that is to say
\begin{align*}
&\forall x\in H_T,&(\varepsilon_\Delta\otimes \id)\circ \delta(x)&=\varepsilon_\Delta(x)1.
\end{align*}\end{itemize}\end{enumerate}

For conilpotent double bialgebras, it is possible to obtain the generalized Eulerian idempotent from a single infinitesimal character $\lambda$:

\begin{prop}
Let $(H,m,\Delta,\delta)$ be a double bialgebra, such that $(H,\Delta)$ is conilpotent. We consider the map $\lambda:H\longrightarrow \K$, defined by $\lambda(1)=0$ and for any $x\in H_+$,
\[\lambda(x)=\sum_{k=1}^\infty \frac{(-1)^{k-1}}{k}\varepsilon_\delta^{\otimes k}\circ \overline{\Delta}_k(x).\]
Then the generalized Eulerian idempotent of $(H,m,\Delta)$ is given by
\[e=(\lambda \otimes \id)\circ \delta.\]
\end{prop}

\begin{proof}
See \cite{Foissy40}. 
\end{proof}

In the case of finite topologies, it is possible to inductively compute this infinitesimal character: 

\begin{prop} 
Let $T$ be a finite topology, associated to the quasi-order $\leq_T$ on the set $E$. We denote by $\min(T)$ the set of classes of $\sim_T$ which are minimal for the order $\overline{\leq_T}$.  
We define $\Upsilon(T)\in \mathbb{Z}[X,X^{-1}]$ by the following:
\begin{align*}
\Upsilon(T)&=\begin{cases}
\displaystyle \frac{1}{X} \mbox{ if }T=1,\\
\displaystyle \sum_{\emptyset\subsetneq I\subseteq \min(T)}X \Upsilon(T_{\mid E\setminus I})\mbox{ otherwise}.
\end{cases}
\end{align*}
Then, for any nonempty finite topology $T$, $\Upsilon(T)\in \mathbb{Z}[X]$ and
\[\lambda(X)=\int_{-1}^0 \Upsilon(T)(t)\mathrm{d}t.\]
\end{prop}

\begin{proof}
Let $T$ be a nonempty finite topology, associated to the quasi-order $\leq_T$. Then
\begin{align*}
\lambda(T)&=\sum_{k=1}^\infty \frac{(-1)^{k-1}}{k} \sum_{\substack{f:E\twoheadrightarrow [k],\\ x\leq_T y \Longrightarrow f(x)\leq f(y)}} \varepsilon_\delta(T_{\mid f^{-1}(1)})\ldots  \varepsilon_\delta(T_{\mid f^{-1}(k)})\\
&=\sum_{k=1}^\infty  \sum_{\substack{f:E\twoheadrightarrow [k],\\ x<_T y \Longrightarrow f(x)< f(y),\\ x\sim_Ty\Longrightarrow f(x)=f(y)}}\frac{(-1)^{k-1}}{k}.
\end{align*}
We then put, for any finite topology $T$ on a set $E$,
\begin{align}
\label{defupsilon}\Upsilon(T)=\sum_{k=0}^\infty \sum_{\substack{f:E\twoheadrightarrow [k],\\ x<_T y \Longrightarrow f(x)< f(y),\\ x\sim_Ty\Longrightarrow f(x)=f(y)}} X^{k-1},
\end{align}
Note that by convention, $\Upsilon(1)=\dfrac{1}{X}$. Then
\[\lambda(T)=\int_{-1}^0 \Upsilon(T)(t)\mathrm{d}t.\]
Let $f:E\twoheadrightarrow [k]$ be such that for any $x,y\in E$ such that $x<_T y$, then $f(x)<f(y)$ and for any $z,t\in E$ such that $z\sim_T t$, then $f(z)=f(t)$. Then $f^{-1}(1)$ is a nonempty subset of $\min(T)$, and we obtain that
\begin{align*}
\Upsilon(T)&=\sum_{\emptyset\subsetneq I\subseteq \min(T)}\sum_{k=1}^\infty
\sum_{\substack{f:E\setminus I\longrightarrow \{2,\ldots,k\},\\ x<_T y\Longrightarrow f(x)<f(y)\\ x\sim_Ty\Longrightarrow f(x)=f(y)}}X^{k-2+1}\\
&= \sum_{\emptyset\subsetneq I\subseteq \min(T)}  \sum_{\emptyset\subsetneq I\subseteq \min(T)}X \Upsilon(P_{\mid E\setminus I}),
\end{align*}
which allows to compute $\Upsilon(T)$ by induction on the number of vertices. 
\end{proof}

\begin{example}
\begin{align*}
\Upsilon(\tun)&=1,&\Upsilon(\tun\tun)&=2X+1,&\Upsilon(\tun\tun\tun)&=6X^2+6X+1,\\
\Upsilon(\tdeux)&=X,&\Upsilon(\ttroisun)=\Upsilon(\ptroisun)&=2X^2+X,&\Upsilon(\ttroisdeux)&=X^2,\\
\Upsilon(\tun\tdeux)&=3X^2+2X,
\end{align*}
and consequently
\begin{align*}
\lambda(\tun)&=1,&\lambda(\tun\tun)&=0,&\lambda(\tun\tun\tun)&=0,\\
\lambda(\tdeux)&=-\frac{1}{2},&\lambda(\ttroisun)=\lambda(\ptroisun)&=\frac{1}{6},&\lambda(\ttroisdeux)&=-\frac{1}{3},\\
\lambda(\tun\tdeux)&=0,
\end{align*}
which finally gives, for $e$ and for the canonical $\pi$-idempotent:
\begin{align*}
e(\tun)&=\tun,&\pi\circ e(\tun)&=\tun,\\
e(\tdeux)&=\tdeux-\frac{1}{2}\tun\tun,&\pi \circ e(\tdeux)&=\tdeux-\frac{1}{2}\tun\tun,\\
e(\ttroisun)&=\ttroisun-\tdeux\tun+\frac{1}{6}\tun\tun\tun,&\pi\circ e(\ttroisun)&=\frac{1}{6}\tun\tun\tun-\tdeux\tun+\frac{1}{2}\ttroisun+\frac{1}{2}\ptroisun,\\
e(\ptroisun)&=\ptroisun-\tdeux\tun+\frac{1}{6}\tun\tun\tun,&\pi\circ e(\ptroisun)&=\frac{1}{6}\tun\tun\tun-\tdeux\tun+\frac{1}{2}\ttroisun+\frac{1}{2}\ptroisun,\\
e(\ttroisdeux)&=\ttroisdeux-\tdeux\tun+\frac{1}{3}\tun\tun\tun,&\pi\circ e(\ttroisdeux)&=\frac{1}{3}\tun\tun\tun-\tdeux\tun+\ttroisdeux.
\end{align*}
Here are the values of $\lambda$ on connected posets of order 4 (it is zero on non connected posets):
\begin{align*}\begin{array}{|c||c|c|c|c|}
\hline T&\tquatreun,\pquatreun&\tquatredeux,\pquatredeux,\pquatrecinq&\tquatrequatre,\pquatrequatre,\pquatresept,\pquatrehuit&\tquatrecinq\\ 
\hline\hline&&&&\\[-2mm]
 \lambda(T)&0&-\dfrac{1}{12}&-\dfrac{1}{6}&-\dfrac{1}{4}\\[-2mm]
&&&&\\
\hline\end{array}\end{align*}

Here are a few examples of order 5:
\begin{align*}
\lambda(\tcinqun)&=-\frac{1}{30},&\lambda(\tcinqhuit)&=\frac{1}{12},&\lambda(\tcinqtreize)&=\frac{3}{20}.
\end{align*}
\end{example}

Let us now give two families of examples.

\begin{prop}\begin{enumerate}
\item For any $n\geq 1$, the $n^{\mathrm{th}}$ ladder is the finite topology associated to the poset $([n],\leq)$:
\begin{align*}
l_1&=\tun,&l_2&=\tdeux,&l_3&=\ttroisdeux,&l_4&=\tquatrecinq,&l_5&=\tcinqquatorze\ldots\\
\end{align*}
If $n\geq 1$, then $\Upsilon(l_n)=X^{n-1}$ and $\lambda(l_n)=\dfrac{(-1)^{n+1}}{n}$. 
\item For any $n\geq 1$, we write
\[\Upsilon(\tun^n)=\sum_{k=0}^{n-1} s_{n,k}X^k. \]
Then $s_{n,k}$ is the number of surjective maps from $[n]$ to $[k+1]$. In particular, $s_{n,n-1}=n!$ and $s_{n,0}=1$. 
\end{enumerate}\end{prop}

\begin{proof}
We use the definition of $\Upsilon$ given by (\ref{defupsilon}). For $l_n$, indexing the vertices from the root to the leaf, 
the unique surjective map to be taken in account in the sum defining $\Upsilon(l_n)$ is $\id_{[n]}$. For $\tun^n$, all surjective maps have to be taken into account.
\end{proof}

\begin{cor}
For any $n\geq 2$, let us denote by $c_n$ the $n^{\mathrm{th}}$  corolla, that is to say the finite topology on $[n]$ given by
\[ \{I\mid I\subseteq[n-1]\}\cup\{[n]\}.\]
Graphically, 
\begin{align*}
c_2&=\tdeux,&c_3&=\ttroisun,&c_4&=\tquatreun,&c_5&=\tcinqun\ldots
\end{align*}
Then, for any $n\geq 2$,  
\[\lambda(c_n)=\sum_{k=0}^{n-2} s_{n-1,k}\frac{(-1)^{k+1}}{k+2}.\]
\end{cor}

\begin{proof}
As $c_n$ has a unique minimal element, 
\[\Upsilon(c_n)=X \Upsilon(\tun^{n-1})=\sum_{k=0}^{n-2}s_{n-1,k}X^{k+1}. \]
The results then follows by integration between $-1$ and $0$. 
\end{proof}

Moreover, due to the form of $\delta(l_n)$ and $\delta(c_n)$, we let the reader check that

\begin{cor}
For any $n\geq 2$,
\begin{align*}
e(l_n)&=\sum_{k=1}^n \sum_{\substack{n=i_1+\ldots+i_k,\\ i_1,\ldots,i_k\geq 1}}\frac{(-1)^{k+1}}{k}l_{i_1}\ldots l_{i_k}\\
&=\sum_{1i_1+\ldots+ni_n=n}
\frac{(-1)^{i_1+\ldots+i_n+1}(i_1+\ldots+i_n-1)!}{i_1!\ldots i_n!}l_1^{i_1}\ldots l_n^{i_n},\\
e(c_n)&=\sum_{i=0}^{n-1}\binom{n-1}{i}\lambda(c_{n-i}) \tun^i c_{n-i},
\end{align*}
where by convention $c_1=\tun$. 
\end{cor}

\bibliographystyle{amsplain}
\bibliography{biblio}

\providecommand{\bysame}{\leavevmode\hbox to3em{\hrulefill}\thinspace}
\providecommand{\MR}{\relax\ifhmode\unskip\space\fi MR }
\providecommand{\MRhref}[2]{%
  \href{http://www.ams.org/mathscinet-getitem?mr=#1}{#2}
}
\providecommand{\href}[2]{#2}
\begin{thebibliography}{10}

\bibitem{Alexandroff1937}
P.~{Alexandroff}, \emph{{Diskrete R\"aume.}}, {Rec. Math. Moscou, n. Ser.}
  \textbf{2} (1937), 501--519 (German).

\bibitem{BFT}
B.~Bellingeri, E.~Ferrucci, and N.~Tapia, \emph{Branched it\^o formula and
  natural it\^o-stratonovich isomorphism}, arXiv:2312.04523v1.

\bibitem{CP21}
P.~Cartier and F.~Patras, \emph{Classical {H}opf algebras and their
  applications}, Springer, 2021.

\bibitem{ebrahimi2015flows}
K.~Ebrahimi-Fard, S.~J.A. Malham, F.~Patras, and A.~Wiese, \emph{Flows and
  stochastic taylor series in {I}t{\^o} calculus}, Journal of Physics A:
  Mathematical and Theoretical \textbf{48} (2015), no.~49, 495202.

\bibitem{ebrahimi2015exponential}
K.~Ebrahimi-Fard, S.J.A. Malham, F.~Patras, and A.~Wiese, \emph{The exponential
  {L}ie series for continuous semimartingales}, Proc. R. Soc. A, vol. 471, The
  Royal Society, 2015, p.~20150429.

\bibitem{EFG}
Kurusch Ebrahimi-Fard and Li~Guo, \emph{Mixable shuffles, quasi-shuffles and
  {Hopf} algebras.}, J. Algebr. Comb. \textbf{24} (2006), no.~1, 83--101.

\bibitem{F17}
L.~Foissy, \emph{Algebraic structures associated to operads}, AxXiV:1702.05344
  (2017).

\bibitem{FoissyMalvenuto}
L.~Foissy and C.~Malvenuto, \emph{The {H}opf algebra of finite topologies and
  t-partitions}, Journal of Algebra \textbf{438} (2015), 130--169.

\bibitem{FMP}
L.~Foissy, C.~Malvenuto, and F.~Patras, \emph{Infinitesimal and {B}-infinity
  algebras, finite spaces, and quasi-symmetric functions}, J. Pure Appl.
  Algebra \textbf{220 (6)} (2016), 2434--2458.

\bibitem{foissy2}
L.~Foissy, F.~Patras, and J.-Y. Thibon, \emph{Deformations of shuffles and
  quasi-shuffles}, Annales Inst. Fourier \textbf{66} (2016), 209--237.

\bibitem{Foissy37}
Lo{\"{\i}}c Foissy, \emph{Commutative and non-commutative bialgebras of
  quasi-posets and applications to {E}hrhart polynomials}, Adv. Pure Appl.
  Math. \textbf{10} (2019), no.~1, 27--63.

\bibitem{Foissy40}
\bysame, \emph{Bialgebras in cointeraction, the antipode and the eulerian
  idempotent}, arXiv:2201.11974, 2022.

\bibitem{GJ94}
E.~Getzler and J.D.S. Jones, \emph{Operads, homotopy algebra and iterated
  integrals for double loop spaces}, arXiv preprint hep-th/9403055 (1994).

\bibitem{Guo}
Li~Guo, \emph{An introduction to {Rota}-{Baxter} algebra}, Surv. Mod. Math.,
  vol.~4, Somerville, MA: International Press; Beijing: Higher Education, 2012.

\bibitem{Guo2}
Li~Guo and William Keigher, \emph{Baxter algebras and shuffle products}, Adv.
  Math. \textbf{150} (2000), no.~1, 117--149.

\bibitem{Hoffman}
M.~E. Hoffman, \emph{Quasi-shuffle products}, J. Algebraic Combin. \textbf{11}
  (2000), no.~1, 49--68.

\bibitem{Hoffman3}
\bysame, \emph{{Quasi-shuffle algebras and applications}}, Algebraic
  combinatorics, resurgence, moulds and applications (CARMA). Volume 2, Berlin:
  European Mathematical Society (EMS), 2020, pp.~327--348.

\bibitem{Hoffman2}
M.~E. Hoffman and K.~Ihara, \emph{Quasi-shuffle products revisited}, J. Algebra
  \textbf{481} (2017), 293--326.

\bibitem{Loday2010}
Jean-Louis Loday and Mar\'{\i}a Ronco, \emph{Combinatorial {H}opf algebras},
  Quanta of maths, Clay Math. Proc., vol.~11, Amer. Math. Soc., Providence, RI,
  2010, pp.~347--383.

\bibitem{Mal}
C.~Malvenuto, \emph{Produits et coproduits des fonctions quasi-sym\'etriques et
  de l'alg\`ebre des descentes}, Publications du Laboratoire de Combinatoire et
  d'Informatique Math\'ematique - UQAM \textbf{16} (1994).

\bibitem{MR}
C.~Malvenuto and Ch. Reutenauer, \emph{Duality between quasi-symmetrical
  functions and the solomon descent algebra}, Journal of Algebra \textbf{177}
  (1995), no.~3, 967--982.

\bibitem{novelli2}
J-C Novelli, F.~Patras, and J-Y Thibon, \emph{Natural endomorphisms of
  quasi-shuffle {H}opf algebras}, Bull. Soc. math. France \textbf{141} (2013),
  no.~1, 107--130.

\bibitem{patras1993decomposition}
F.~Patras, \emph{La d{\'e}composition en poids des alg\`ebres de {H}opf},
  Annales de l'institut Fourier \textbf{43} (1993), no.~4, 1067--1087.

\bibitem{patras1994algebre}
\bysame, \emph{L'alg{\`e}bre des descentes d'une big{\`e}bre gradu{\'e}e},
  Journal of Algebra \textbf{170} (1994), no.~2, 547--566.

\bibitem{patras1992}
\bysame, \emph{Homoth\'eties simpliciales}, PhD thesis, Univ. Paris 7, Jan.
  1992.

\bibitem{PR99}
F.~Patras and Ch. Reutenauer, \emph{Higher lie idempotents}, Mosc. Math. J.
  (1999), no.~222, 51--64.

\bibitem{PR02}
\bysame, \emph{Lie representations and an algebra containing solomon's}, J.
  Alg. Comb. (2002), no.~16, 301--314.

\bibitem{PR02b}
\bysame, \emph{On dynkin and klyachko idempotents in graded bialgebras}, Mosc.
  Math. J. \textbf{3} (2002), no.~28, 560--579.

\bibitem{Reutenauer}
Ch. Reutenauer, \emph{Free {L}ie algebras}, Oxford University Press, 1993.

\bibitem{Takeuchi1971}
M.~{Takeuchi}, \emph{{Free Hopf algebras generated by coalgebras}}, {J. Math.
  Soc. Japan} \textbf{23} (1971), 561--582.

\end{thebibliography}

\end{document}